\documentclass[10pt,reqno]{amsart}
\usepackage{geometry}
\usepackage{xcolor}
\usepackage{amsthm}
\usepackage{amssymb}
\usepackage{hyperref}
\usepackage{arydshln}
\usepackage{bbm}
\usepackage{mathrsfs}

\usepackage[utf8]{inputenc} 
\usepackage[T1]{fontenc}

\makeatletter
\numberwithin{equation}{section}
\numberwithin{figure}{section}

\newcommand{\C}{\mathbb{C}}
\newcommand{\R}{\mathbb{R}}
\newcommand{\N}{\mathbb{N}}

\newcommand{\F}{\mathcal{F}}

\renewcommand{\P}{\mathbb{P}}
\newcommand{\E}{\mathbb{E}}
\newcommand{\e}{\varepsilon}

\newcommand{\1}{\mathbbm{1}}

\newtheorem{Theorem}{Theorem}[section]
\newtheorem{Proposition}[Theorem]{Proposition}\newtheorem{Corollary}[Theorem]{Corollary}\newtheorem{Lemma}[Theorem]{Lemma}\newtheorem{Remark}[Theorem]{Remark}\newtheorem{Definition}[Theorem]{Definition}\newtheorem{Example}[Theorem]{Example}

\numberwithin{equation}{section}

\makeatother

\begin{document}

\title[Regular occupation measures of Volterra processes]{Regular occupation measures of Volterra processes}

\author{Martin Friesen}

\address[Martin Friesen]
{School of Mathematical Sciences
\\ Dublin City University
\\ Glasnevin, Dublin 9, Ireland}
\email{martin.friesen@dcu.ie}

\date{\today}



\begin{abstract} {
   We introduce a local non-determinism condition for Volterra It\^{o} processes that captures smoothing properties of possibly degenerate noise. By combining the stochastic sewing lemma with one-step Euler approximations, we first prove the joint space-time regularity for their occupation measure, self-intersection measure, and time marginals for such Volterra It\^{o} processes. As an application, we obtain the space-time regularity of local times and self-intersection times for rough perturbations of Gaussian Volterra processes, and construct a class of non-Gaussian Volterra I\^{o} processes that are $C^{\infty}$-regularising. Secondly, for the particular class of stochastic Volterra equations with H\"older continuous coefficients, using disintegration of measures for their Markovian lifts, we further establish the absolute continuity of finite-dimensional distributions. Finally, we prove the existence, uniqueness, and stability for self-interacting stochastic equations with distributional drifts.}
\end{abstract}

\maketitle

\allowdisplaybreaks

\section{Introduction}

\subsection{General overview}

In the past recent years, stochastic Volterra processes have gained increased attention due to their ability to capture the rough behaviour of sample paths. Such processes provide a flexible framework for modelling rough volatility in Mathematical Finance, see e.g. \cite{MR3494612, MR3778355, MR4188876, MR3805308}. In the absence of jumps, {a general form of such a process in the time-homogeneous case} is given by
\begin{align}\label{eq: CVSDE}
 X_t = g(t) + \int_0^t K_b(t,s)b(X_s)\, \mathrm{d}s + \int_0^t K_{\sigma}(t,s)\sigma(X_{s})\, \mathrm{d}B_s,
\end{align}
where $K_b: [0,T]^2 \longrightarrow \R^{d \times n}$ and $K_{\sigma}: [0,T]^2 \longrightarrow \R^{d \times k}$ denote non-anticipating Volterra kernels, $b: \R^d \longrightarrow \R^n$ the drift, $\sigma: \R^d \longrightarrow \R^{k \times m}$ the diffusion coefficient, $g: [0,T] \longrightarrow \R^d$ {a} $\mathcal{B}([0,T])\otimes \F_0$-measurable random variable, and $(B_t)_{t \geq 0}$ is an $m$-dimensional $(\F_t)_{t\in [0,T]}$-Brownian motion on a filtered probability space $(\Omega, \F, (\F_t)_{t \in [0,T]}, \P)$ satisfying the usual conditions. For regular Volterra kernels, strong existence and uniqueness of \eqref{eq: CVSDE} were studied in \cite{berger_mizel, protter_volterra} for Lipschitz continuous coefficients, and in \cite{PROMEL2023291} for H\"older continuous coefficients. In dimension $d = 1$, strong solutions for H\"older continuous coefficients and fractional kernels were studied in \cite{mytnik, PROMEL2023291b}. Finally, for the weak existence of solutions, we refer to \cite{MR4019885, AbiJaber_branching, PROMEL2023291a}. Note that such processes are typically neither semimartingales nor Markov processes \cite{FGW25}, which introduces new challenges beyond classical Markov process theory.

In applications, stochastic Volterra processes are often confined to a region $D \subset \R^d$, i.e., $X_t \in D$ holds a.s.. For instance, in dimension $d = 1$ with $D = \R_+$, the \textit{Volterra Cox-Ingersoll-Ross process} models the instantaneous {variance} and is defined as the unique nonnegative weak solution of
\begin{align}\label{eq: Volterra CIR}
    X_t = x_0 + \int_0^t K(t-s)\left( b + \beta X_s\right)\, \mathrm{d}s + \sigma \int_0^t K(t-s)\sqrt{X_s}\, \mathrm{d}B_s,
\end{align}
where $x_0,b, \sigma \geq 0$, $\beta \in \R$, and $K \in L_{loc}^2(\R_+)$ is a completely monotone kernel that satisfies an additional regularity condition on its increments, see \cite[Theorem 6.1]{MR4019885}. Observe that $\sigma(x) = \sqrt{x}$ in \eqref{eq: Volterra CIR} is merely H\"older continuous and degenerate at the boundary of the state-space, i.e., $\sigma(0) = 0$. This degeneracy is a generic feature that confines the process to its state space, see \cite{MR4019885} for $D = \R_+^m$, \cite{MR4832220} for convex cones $D$, {and \cite{MR4836009} for compact state-spaces $D \subset \R^d$ in the context of polynomial Volterra processes.} Both features, H\"older continuous coefficients and degeneracy of $\sigma$ at the boundary of its state-space, form core difficulties in the study of \eqref{eq: Volterra CIR}, and general stochastic Volterra equations \eqref{eq: CVSDE}. While for the classical case $K(t) \equiv 1$, several tools are available to study the boundary behaviour in terms of Feller conditions and the regularity of its distribution up to the boundary (see \cite{MR2857449, FJKR22, MR4195178}), much less is known for its Volterra counterpart, and even less for general equations of the form \eqref{eq: CVSDE}. {At this point, it is worth noting the results in \cite{BP24} for regular Volterra kernels.}

In this work, we address the regularity of the one-dimensional distribution $\mathcal{L}(X_t)$, investigate the absolute continuity of finite-dimensional distributions $\mathcal{L}(X_{t_1}, \dots, X_{t_n})$, and study the regularity for the corresponding local and self-intersection time of the Volterra process. For Markov processes, such regularity results are often essential for weak convergence rates in numerical schemes, are frequently studied for irreducibility and stochastic stability, and play an important role in nonparametric inference methods. For stochastic Volterra processes, the absolute continuity of distributions is also used to show that solutions of \eqref{eq: CVSDE} are, in general, not Markov processes \cite{FGW25}. While for smooth coefficients the regularity of $\mathcal{L}(X_t)$ might be studied using the powerful methods based on Malliavin calculus \cite{MR2857449}, such results do not apply to \eqref{eq: CVSDE} with H\"older continuous coefficients. Likewise, while for Markov processes the Chapman-Kolmogorov equations provide a representation of finite-dimensional distributions in terms of their one-dimensional laws, the failure of the Markov property for \eqref{eq: CVSDE} rules out such an approach. Finally, while the regularity of local times was recently studied for Gaussian Volterra processes \cite{BLM23, MR4342752}, Volterra L\'evy processes \cite{MR4488556}, or equations with smooth coefficients driven by fractional Brownian motion \cite{LOU20173643}, the general case of \eqref{eq: CVSDE} has not been addressed in the literature.

\subsection{Regularity of occupation measures, self-intersection measures, and time-marginals}

While our motivation stems from applications to stochastic Volterra equations \eqref{eq: CVSDE}, to allow for additional flexibility in applications, in our main results, we treat the general class of Volterra It\^{o} processes of the form 
\begin{align}\label{eq: V Ito process}
 X_t &= g(t) + \int_0^t K_b(t,s)b_s\, \mathrm{d}s + \int_0^t K_{\sigma}(t,s)\sigma_s\, \mathrm{d}B_s.
\end{align}
The extension from \eqref{eq: CVSDE} to Volterra It\^{o} processes allows us to treat a large class of stochastic processes that includes stochastic Volterra equations with time-dependent coefficients, with random coefficients, and that arise as projections of Markovian lifts of \eqref{eq: CVSDE}.

Let $X$ be a Volterra It\^{o} process of the form \eqref{eq: V Ito process}. Its occupation and self-intersection measures are for Borel sets $A \in \mathcal{B}(\R^d)$ defined by 
\begin{align}\label{eq: no weight}
    \ell_t(A) = \int_0^t \1_A(X_r)\, \mathrm{d}r \quad \text{and} \quad 
    G_t(A) = \int_0^t \int_0^t \1_A(X_{r_2} - X_{r_1})\, \mathrm{d}r_1\mathrm{d}r_2.
\end{align}
Their densities with respect to the Lebesgue measure, provided they exist, are called the local time and the self-intersection local time, respectively. Formally, they are given by $\ell_t(x) = \int_0^t \delta_x(X_r)\, \mathrm{d}r$ and similarly $G_t(x) = \int_0^t \int_0^s \delta_x(X_{r_2} - X_{r_1})\, \mathrm{d}r_1 \mathrm{d}r_2$. Such densities encode information about the roughness and non-determinism of the sample paths. Indeed, the function $\ell_t$ is typically smooth for highly irregular paths, while for regular paths $\ell_t$ may not even exist \cite{MR350833}. 

To establish the existence and regularity of these densities, we propose a local non-determinism condition that extends the definitions of \cite{MR4488556, MR4342752} to Volterra It\^{o} processes, and covers cases where $\sigma_t$ may be degenerate (see, e.g., \eqref{eq: Volterra CIR} where $\sigma_t = \sqrt{X_t}$). Namely, we suppose that there exist constants $H>0$, $C_T > 0$, and a progressively measurable process $\rho: \Omega \times [0,T] \longrightarrow [0,1]$ such that for all $s,t \in [0,T]$ with $s \leq t$ and $|\xi| = 1$,
\begin{align}\label{eq: local nondeterminism intro}
        \int_s^t |\sigma_r^{\top}K_{\sigma}(t,r)^{\top}\xi|^2 \mathrm{d}r \geq C_T (t-s)^{2H}\rho_s^2 \quad \text{a.s.}
\end{align}
The newly introduced process $(\rho_t)_{t \in [0, T]}$ reflects possible degeneracies of the noise at the boundary of the state space. However, if $\sigma_t \sigma_t^{\top}$ is a.s. uniformly bounded from below, then we can set $\rho \equiv 1$ and recover the classical definition of local non-determinism.

Based on a combination of the stochastic sewing lemma with one-step Euler approximations in the spirit of \cite{MR2668905, FJ2022, MR3885546}, we derive in Section 3 the regularity in Fourier-Lebesgue spaces (see Section 2 for precise definitions) for the weighted analogues of $\ell_t$ and $G_t$, defined as
\begin{align}\label{eq: weighted measures}
    \ell^{\delta}_t(A) = \int_0^t \rho_r^{\delta}\1_A(X_r)\, \mathrm{d}r \quad \text{and} \quad 
    G^{\delta}_t(A) = \int_0^t \int_0^t \rho_r^{\delta} \rho_s^{\delta}\1_A(X_s - X_r)\, \mathrm{d}r\mathrm{d}s,
\end{align}
where $\delta \geq 0$ denotes a weight parameter and $\rho$ stems from the local non-determinism condition \eqref{eq: local nondeterminism intro}. In particular, if $\delta = 0$, we recover the classical notions of occupation and self-intersection measures. {Finally, for classical diffusion processes in dimension $d = 1$ with $K \equiv 1$, the choice $\rho_r = 1 \wedge \sigma(X_r)^2 = 1 \wedge \frac{d}{dr}\langle X\rangle_r$ reduces to the definition of the local time via the Tanaka's formula with $\delta = 1$.}  

As an illustration of our results, we prove space-time H\"older continuity of the local time and self-intersection time for perturbations $X_t = \psi_t + B_t^H$ of the fractional Brownian motion $B^H$ with rough perturbations $\psi_t = g(t) + \int_0^t K_b(t,s)b_s\, \mathrm{d}s$. Such a result complements the regularity derived in \cite{BLM23} for perturbations $\psi_t$ of finite variation. Secondly, we construct a general class of Volterra It\^{o} processes that are $C^{\infty}$-regularising in the sense that their local times a.s. belong to the space of smooth test functions $\mathcal{D}(\R^d)$. This extends the class of $C^{\infty}$-regularising processes from the strictly Gaussian frameworks studied in \cite{MR4342752} to general Volterra It\^{o} processes.  

Concerning the regularity of time-marginals of \eqref{eq: V Ito process}, define the weighted law $\mu_t^{\delta}(A) = \E\left[ \rho_t^{\delta} \1_A(X_t) \right]$, where $\delta \geq 0$. Here, $\delta = 0$ corresponds to the standard, unweighted law of $X_t$, while in view of \eqref{eq: local nondeterminism intro}, the factor $\rho_t^{\delta}$ suppresses for $\delta > 0$ those regions where the noise vanishes. We prove that $\mu_t^{\delta}$ exhibits dimension-dependent regularity in Fourier-Lebesgue spaces, and dimension-independent regularity in the Besov space $B_{1,\infty}^{\lambda}(\R^d)$. In particular, if $\sup_{t \in [0,T]}\E[\rho_t^{-1}] < \infty$, then we show that the law $\mathcal{L}(X_t)$ is absolutely continuous and its density belongs to $B_{1,\infty}^{\lambda}(\R^d)$. 

Having in mind stochastic Volterra processes \eqref{eq: CVSDE} with \textit{merely H\"older continuous} and possibly degenerate diffusion coefficients, the local non-determinism condition \eqref{eq: local nondeterminism intro} takes the form 
\[
     \int_{s}^{t} | \sigma(X_r)^{\top}K_{\sigma}(t,r)^{\top}\xi|^2 \mathrm{d}r \geq C_* (t-s)^{2H}\sigma_*(X_s)^2
\]
for all $t \in (0,T]$, $s \in (0,t]$, and $|\xi| = 1$, where $C_*, H > 0$ are constants, and $0 \leq \sigma_* \in C^{\theta}(\R^d)$ with $\theta \in (0,1]$ and $\sup_{x} \sigma_*(x) \le 1$ denotes a spatial weight function that captures the behaviour of the diffusion coefficient up to the boundary of its state space. A natural choice for this weight is $\sigma_*(x) = 1 \wedge \lambda_{\min}(\sigma(x)\sigma(x)^{\top})$, where $\lambda_{\min}$ denotes the smallest eigenvalue of $\sigma \sigma^{\top}$. In Section 4, we derive the regularity of the corresponding local time, self-intersection time, and the weighted law $\mu_t^{\delta}(A) = \E\left[ \sigma_*(X_t)^{\delta} \1_A(X_t) \right]$ as particular cases of \eqref{eq: V Ito process}. Afterwards, we address the absolute continuity of the finite-dimensional distributions $(X_{t_1}, \dots, X_{t_n})$. To overcome the lack of the Markov property, we develop a method based on Markovian lifts combined with the disintegration of measures, which allows us to recover a weak form of the Chapman-Kolmogorov equations.

\subsection{Stochastic equations with distributional self-interactions}

In the last part of this work, we focus on the solution theory for self-interacting processes of the form
\begin{align}\label{eq: selfintersecting Volterra diffusion process}
    X_t = x_0 + \int_0^t \int_0^t b(X_s - X_r)\, \mathrm{d}r \mathrm{d}s + Z_t,
\end{align}
where $(Z_t)_{t \geq 0}$ is a continuous stochastic process. Such equations appear, e.g., as models for polymer growth, and in the broader theory of self-interacting diffusions (see, for instance, \cite{MR1165516, MR2952088, MR2451570} for Brownian Polymer models, \cite{MR1883716, MR1348356} for self-intersecting diffusions, and \cite{MR4894548} for extensions towards models driven by the fractional Brownian motion). Within these models, $Z$ captures the random thermal fluctuations (kinetic kicks) imparted by the surrounding medium, $b$ models the self-interaction potential dictating how the polymer is actively repelled from or attracted to regions it has already visited, while $(X_t)_{t \geq 0}$ describes the macroscopic spatial position of the polymer's growing tip at either the contour length or time $t$.

The choice of the interaction kernel $b$ reflects the microscopic physics of the system. In the context of polymer physics, the \textit{excluded volume effect} states that no two segments of a polymer chain can occupy the exact same point in space. To model such self-avoidance, interactions must act instantaneously when the path intersects itself (i.e. when $X_s = X_r$), and ideally forces $b$ to take the form of a spatial gradient $\nabla \delta_0$ of the Dirac distribution. Power-law kernels of the form $b(x) = \nabla |x|^{-\alpha}$ with $\alpha \in (0,d)$ provide a class of singular drifts with nonlocal interactions, while regular kernels $b$, e.g. Gaussian kernels, are also considered in the literature. While for regular interaction kernels $b$, equation \eqref{eq: selfintersecting Volterra diffusion process} can be studied by traditional methods, in this work, we focus on the case of singular interaction kernels $b$ that belong to a space of distributions with possibly negative regularity. 

The study of \eqref{eq: selfintersecting Volterra diffusion process}, even the definition of the integral therein, requires regularisation by noise and hence rough drivers $Z$. Such regularisation by noise phenomena was first observed in \cite{MR0336813} for the simple initial value problem 
\begin{align}\label{eq: rough ODE}
    u(t) = x_0 + \int_0^t b(u(s))\, \mathrm{d}s + Z_t
\end{align}
with $Z$ being the Brownian motion, see also \cite{MR0568986, MR1864041, MR2172887, MR2117951}. The case where $Z$ is a fractional Brownian motion was treated in \cite{MR1934157, MR2073441}. While in all of these results, the equation is solved in a pathwise strong sense, \cite{MR2377011} establishes \textit{path-by-path uniqueness} which treats the equation as an ordinary differential equation perturbed by a single path $z$ sampled from the Brownian motion $Z$. More recently, in \cite{CATELLIER20162323} the authors establish pathwise regularization by noise phenomena for \eqref{eq: rough ODE} treated as a deterministic equation with $Z$ sampled from the fractional Brownian motion with drift $b$ that belongs to a Besov space with negative regularity (or the Fourier Lebesgue space defined in Section \ref{sec: preliminaries}). The authors consider controlled solutions of the form $u(t) = \theta_t + Z_t$ where $\theta$ formally satisfies the corresponding nonlinear Young equation. Further extensions have been studied, e.g., in \cite{BLM23, GG22, MR4488556, MR4342752}, while an overview of the nonlinear Young integral and related equations is given in \cite{G21}.

In this work, we study \eqref{eq: selfintersecting Volterra diffusion process} for distributional drifts $b \in \mathcal{S}'(\R^d)$ driven by a rough process $Z$. Our method is based on the ansatz $X_t = \theta_t + Z_t$ with $\theta$ solving
\begin{align}\label{eq: 2}
    \theta_t = u_0 + \int_0^t \int_0^t b(\theta_{r_2} - \theta_{r_1} + Z_{r_2} - Z_{r_1})\, \mathrm{d}r_1 \mathrm{d}r_2,
\end{align}
where the integral on the right-hand side has to be understood as a two-parameter Young integral with respect to the two-parameter self-intersection measure $G_{t_1,t_2}(A) = \int_{0}^{t_2}\int_{0}^{t_1} \1_{A}(X_{r_2} - X_{r_1})\, \mathrm{d}r_1 \mathrm{d}r_2$. Under the assumption that $Z$ has a sufficiently smooth self-intersection time, we prove the existence, uniqueness, and stability of solutions to \eqref{eq: 2} and hence \eqref{eq: selfintersecting Volterra diffusion process}. As a particular example, the equation
\[
    X_t = x_0 + \theta \int_0^t \int_0^t \nabla \delta_0(X_s - X_r)\, \mathrm{d}r\mathrm{d}s + B_t^H
\]
with a fractional Brownian motion $B^H$ with Hurst parameter $H$, and $\theta \neq 0$, has a unique solution whenever $H < \frac{1}{d+4}$. The choices $b(x) = \delta_0$ in dimension $d=1$, $b(x) = \theta \nabla \chi(x) x^{-\alpha}$ with general $d \geq 1$, and $b(x) = \theta \chi(x) \mathrm{sgn}(x)$ in $d=1$ with a smooth and compactly supported cutoff function $\chi$ satisfying $\chi \equiv 1$ in a neighourhood of the origin, provide other classes of examples covered by this work.

\subsection{Structure of the work}

This work is organised as follows. In Section 2, we introduce the function spaces used in this work, recall the classical stochastic sewing lemma, and prove a simple criterion that provides a priori bounds on the Fourier transform of the weighted local time. In Section 3, we study the regularity of distributions for the general class of Volterra It\^{o} processes and provide some examples. Section 4 is dedicated to the particular application to stochastic Volterra processes of the form \eqref{eq: CVSDE}, where, in particular, we prove that its f.d.d. are absolutely continuous. Section 5 concerns the regularisation by noise phenomenon for stochastic equations with distributional self-intersections.

\section{Preliminaries}
\label{sec: preliminaries}

\subsection{Function spaces}

Let $\mathcal{S}(\R^d)$ be the space of Schwartz functions over $\R^d$ and let $\mathcal{S}'(\R^d)$ be its dual space of tempered distributions. The Fourier transform on $\mathcal{S}(\R^d)$ and its extension onto $\mathcal{S}'(\R^d)$ is denoted by $\mathcal{F}$. Sometimes we also write $\widehat{f} = \mathcal{F}f$. Note that $\mathcal{F}$ is an isomorphism on $\mathcal{S}(\R^d)$ as well as on $\mathcal{S}'(\R^d)$ with inverse denoted by $\mathcal{F}^{-1}$. Let $L^p(\R^d)$ be the standard Lebesgue space on $\R^d$ with $p \in [1,\infty]$. We denote its norm by $\|\cdot \|_{L^p}$, and let $L_{loc}^p(\R^d) \subset L^p(\R^d)$ be its subspace of locally $p$-integrable functions. 

The convolution of $f \in \mathcal{S}(\R^d)$ and $g \in \mathcal{S}'(\R^d)$ is defined by $f \ast g = \mathcal{F}^{-1}(\widehat{f} \cdot \widehat{g})$ in $\mathcal{S}'(\R^d)$. Young's inequality plays an essential role in determining if this convolution is more regular, i.e., can be extended to a bilinear mapping on different scales of Banach spaces. {Among such, the the scale of Besov spaces $B_{p,q}^s(\R^d)$ plays a central role. To define the latter, let $(\varphi_j)_{j \geq 0}$ be a smooth dyadic partition of unity, i.e. a collection of smooth functions with values in $[0,1]$ such that $\mathrm{supp}(\varphi_0) \subseteq \{ \xi \in \R^d \ : \ |\xi| \leq 2\}$ and $\mathrm{supp}(\varphi_j) \subseteq \{ \xi \in \R^d \ : \ 2^{j-1} \leq |\xi| \leq 2^{j+1}\}$ for $j \geq 1$, $\sup_{x \in \R^d}2^{j |\alpha|}|D^{\alpha}\varphi_j(x)| < \infty$ for any multi-index $\alpha \in \N_0^d$, and $\sum_{j=0}^{\infty}\varphi_j = 1$. For any $f \in \mathcal{S}'(\R^d)$ we define the dyadic Littlewood-Paley blocks by $\Delta_j f := \mathcal{F}^{-1}(\varphi_j \cdot \mathcal{F}(f)) \in \mathcal{S}(\R^d)$. The Besov space $B_{p,q}^s(\R^d)$ with $1 \leq p,q \leq \infty$, and $s \in \R$ consists of all $f \in \mathcal{S}'(\R^d)$ with finite norm
\[
    \| f\|_{B_{p,q}^s} = \left( \sum_{j=0}^{\infty} 2^{jsq} \| \Delta_j f\|_{L^p(\R^d)}^q\right)^{\frac{1}{q}}, \qquad 1 \leq q < \infty,
\]
and if $q = \infty$ with the obvious modification $\|f\|_{B_{p, \infty}^s} = \sup_{j \geq 0}2^{js}\| \Delta_j f\|_{L^p(\R^d)}$. For further details and properties on these spaces, we refer to \cite{MR3243734, MR1163193}, while a Young inequality for convolutions in this scale of Banach spaces was obtained in \cite[Theorem 2.1, Theorem 2.2]{MR4339468}.}

As a particular case, let us denote by
\[
    \mathcal{C}^s(\R^d) := B_{\infty,\infty}^s(\R^d)
\]
the H\"older-Zygmund space. For $s > 0$ such that $s \not \in \N_0$ the H\"older-Zygmund space $\mathcal{C}^s(\R^d)$ coincides with the space of bounded H\"older continuous functions $C_b^{s}(\R^d)$. For $s \in \N_0$, $\mathcal{C}^s(\R^d)$ is larger than the classical H\"older space. Similarly, let $H_p^s(\R^d)$ be the fractional Sobolev space with $1 < p < \infty$. Due to the Littlewood-Paley theory, these spaces can be characterised through the Fourier multiplier {$\langle \cdot \rangle: \R^d \longrightarrow \R_+, \langle \xi \rangle = (1+|\xi|^2)^{1/2}$} by
\[
    H_p^s(\R^d) = \{ f \in S'(\R^d) \ : \ \mathcal{F}^{-1}( \langle \cdot \rangle^s \widehat{f}) \in L^p(\R^d) \},
\]
where $\| f\|_{H_p^s} = \|\mathcal{F}^{-1}( \langle \cdot \rangle^s \widehat{f})\|_{L^p(\R^d)}$ defines an equivalent norm. Note that $H_p^k(\R^d) = W^{k,p}(\R^d)$ denotes the classical Sobolev space when $k \in \N_0$. 

To capture the regularity of local times, we introduce, similarly to \cite{CATELLIER20162323}, the Fourier-Lebesgue spaces 
\[
 \mathcal{F}L_p^s(\R^d)
 = \{ f \in \mathcal{S}'(\R^d) \ : \ \widehat{f} \in L_{loc}^p(\R^d) \ \text{ and } \ \| f\|_{\mathcal{F}L_p^s} < \infty \}
\]
where $p \in [1,\infty]$, $s \in \R$, and the norm is given by $\| f\|_{\mathcal{F}L_p^s} = \| \langle \cdot\rangle^s  \widehat{f}\|_{L^p(\R^d)}$. By H\"older inequality, one can verify that these spaces satisfy for all $1 \leq q \leq p \leq \infty$ and $s \in \R$
\[
 \mathcal{F}L_p^s(\R^d) \hookrightarrow \mathcal{F}L_q^{s - \left( \frac{1}{q} - \frac{1}{p}\right)d - \e}(\R^d), \qquad \forall \e > 0.
\]
Finally, given $p, p_0,p_1 \in [1,\infty]$ such that $\frac{1}{p_0} + \frac{1}{p_1} = \frac{1}{p}$, $s_0, s_1 \in \R$, $f \in \mathcal{F}L_{p_0}^{s_0}(\R^d)$, and $g \in \mathcal{F}L_{p_1}^{s_1}(\R^d)$, then the convolution $f \ast g$ is well-defined in $\mathcal{F}L_{p}^{s_0+s_1}(\R^d)$, and the following Young inequality holds
\begin{align}\label{eq: Young FL}
  \| f \ast g\|_{\mathcal{F}L_p^{s_0 + s_1}} \leq \| f\|_{\mathcal{F}L_{p_0}^{s_0}}\| g\|_{\mathcal{F}L_{p_1}^{s_1}}.
\end{align}
Of particular interest are the embeddings of the Fourier-Lebesgue space into the H\"older-Zygmund scale and the fractional Sobolev spaces as stated below.

\begin{Lemma}\label{lemma: Fourier Lebesgue embedding}
    If $s \in \R$ and $1 < p \leq 2$, then 
    \[
        \mathcal{F}L_p^s(\R^d) \hookrightarrow H_{\frac{p}{p-1}}^s(\R^d) \ \text{ and } \ 
        H_{p}^s(\R^d) \hookrightarrow \mathcal{F}L_{\frac{p}{p-1}}^s(\R^d),
    \]
    while for $p = 1$ we obtain $\mathcal{F}L_1^s(\R^d) \hookrightarrow \mathcal{C}^s(\R^d)$.
\end{Lemma}
\begin{proof}
 For $\sigma \in \R$ let $I^{\sigma}f = \mathcal{F}^{-1}(\langle \cdot \rangle^{\sigma} \widehat{f})$. It follows from \cite[p. 58-59]{MR3024598} that $I^{\sigma}: H^s_{q} \longrightarrow H^{s-\sigma}_{q}(\R^d)$ is a continuous linear isomorphism for all $s,\sigma \in \R$, $1 < p \leq 2$, and $q = \frac{p}{p-1}$. Let $f \in \mathcal{F}L_p^s(\R^d)$, then $g(\xi) := \langle \xi \rangle^{s}\widehat{f}(\xi)$ belongs to $L^p(\R^d)$ and hence by the Hausdorff-Young inequality $\mathcal{F}^{-1}g(x) = \widehat{g}(-x)$ belongs to $L^{q}(\R^d) \hookrightarrow H_{q}^0(\R^d)$. Hence we obtain $f = I^{-s}(I^sf) = I^{-s}\mathcal{F}^{-1}g \in H_{q}^s(\R^d)$. Conversely, if $f \in H_p^{s}(\R^d)$ with $1 < p \leq 2$, then we obtain
 \begin{align*}
    \|f\|_{\mathcal{F}L_{p/(p-1)}^s} 
    = \left\| \mathcal{F}\mathcal{F}^{-1}( \langle\cdot\rangle^s \widehat{f})\right\|_{L^{p/(p-1)}(\R^d)} 
    \lesssim \left\| \mathcal{F}^{-1}( \langle \cdot \rangle^s \widehat{f} )\right\|_{L^{p}(\R^d)} = \|f\|_{H_{p}^s}.
 \end{align*}
 
 Now let $p = 1$. Then $g \in L^1(\R^d)$ and hence $\mathcal{F}^{-1}g \in C(\R^d)$ vanishes at infinity. For the H\"older-Zygmund norm, we obtain
 \begin{align*}
     \| f\|_{\mathcal{C}^s} = \sup_{j \geq 0}2^{js}\| \mathcal{F}^{-1}(\varphi_j \widehat{f})\|_{L^{\infty}}
     \leq \sup_{j \geq 0}2^{js} \| \varphi_j \widehat{f}\|_{L^1}
     \lesssim \|\langle \cdot \rangle^s \widehat{f}\|_{L^1} = \|f\|_{\mathcal{F}L_1^s}.
 \end{align*}
 This proves the assertion.
\end{proof}

Let $(E, \| \cdot \|_{E})$ be a Banach space. Denote by $C([0,T]; E)$ the space of continuous functions from $[0,T]$ to $E$ equipped with the supremum norm $\|\cdot\|_{\infty}$. When $\alpha \in (0,1)$, then $C^{\alpha}([0,T]; E)$ denotes the space of $\alpha$-H\"older continuous functions equipped with the norm $\|f\|_{C^{\alpha}([0,T]; E)} = \|f\|_{\infty} + [f]_{C^{\alpha}([0,T]; E)}$ where
\[
    [f]_{C^{\alpha}([0,T]; E)} = \sup_{s \neq t}\frac{\|f_t - f_s\|_E}{|t-s|^{\alpha}}.
\]

\subsection{Stochastic sewing techniques}

For $T > 0$, and $n \in \N$ define 
\[
 \Delta_{T}^n = \{ (t_1,\dots, t_n) \in [0,T]^n\ : \ t_1 \leq \dots \leq t_n \}.
\]
Then $\Delta_T^1 = [0,T]$ and $\Delta_T^2 = \{ (s,t) \in [0,T]^2 \ : \ s \leq t \}$. Let $E$ be a Banach space. For a function $f: [0,T] \longrightarrow E$ we define a new function on $\Delta_{T}^2$ by $f_{s,t} = f_t - f_s$. Likewise, for a function $g: \Delta_{T}^2 \longrightarrow E$ we define a new function on $\Delta_{T}^3$ by setting $\delta_r g_{s,t} = g_{s,t} - g_{s,r} - g_{r,t}$ where $(s,r,t) \in \Delta_{T}^3$. 

Let $(E, \| \cdot \|_{E})$ be a Banach space. For $p \in [1,\infty]$ we let $L^p(\Omega, \P; E)$ be the standard $L^p$-space of $E$-valued random variables defined over a probability space $(\Omega, \F, \P)$. The corresponding norm is denoted by $\| X\|_{L^p(\Omega; E)} = \left( \E[\|X\|_E^p]\right)^{1/p}$ when $p \in [1,\infty)$ with obvious modifications for $p = \infty$. For $E = \R^d$ or $E = \C^d$, we simply write $L^p(\Omega, \P)$ and $\| \cdot\|_{L^p(\Omega)}$ for its norm. The following stochastic sewing lemma was shown in \cite{MR4089788}.
\begin{Lemma}\label{lemma: swl}
 Fix $T > 0$ and $p \geq 2$. Let $(\Omega, \F, (\F_t)_{t \in [0,T]}, \P)$ be a probability space with the usual conditions. Suppose $A: \Delta_{T}^2 \longrightarrow \R^d$ is a stochastic process in $L^p(\Omega)$ such that $A_{s,s} = 0$ and $A_{s,t}$ is $\F_t$-measurable for all $(s,t) \in \Delta_{T}^2$. Assume that there exist constants $\varkappa_1 > 1$, $\varkappa_2 > \frac{1}{2}$, and $C_1,C_2 \geq 0$ such that 
 \[
  \left\| \E\left[ \delta_r A_{s,t}\ | \ \F_s\right] \right\|_{L^p(\Omega)} \leq C_1|t-s|^{\varkappa_1}
 \]
 and
 \[
  \| \delta_r A_{s,t} \|_{L^p(\Omega)} \leq C_2|t-s|^{\varkappa_2}
 \]
 hold for all $(s,r,t) \in \Delta_{T}^3$. Then there exists a unique (up to modifications) $(\F_t)_{t \in [0,T]}$-adapted process $\mathcal{A}: [0,T] \longrightarrow \R^d$ with values in $L^p(\Omega, \P)$ such that $\mathcal{A}_{0} = 0$, and there exists another constant $c > 0$ that only depends on $\varkappa_1, \varkappa_2, d, p$ satisfying
 \[
  \left\| \mathcal{A}_{s,t} - A_{s,t} \right\|_{L^p(\Omega)} \leq c\left(C_1|t-s|^{\varkappa_1} + C_2 |t-s|^{\varkappa_2}\right)
 \]
 and
 \[
  \left\| \E\left[ \mathcal{A}_{s,t} - A_{s,t}\ | \ \F_s\right]\right\|_{L^p(\Omega)} \leq cC_1|t-s|^{\varkappa_1}.
 \]
 Finally, for all $(s,t) \in \Delta_{T}^2$ and each partition $\mathcal{P}_{s,t}$ of $[s,t]$, we have 
 \begin{align*}
  \lim_{|\mathcal{P}_{s,t}| \to 0}\left\| \sum_{[u,v] \in \mathcal{P}_{s,t}} A_{u,v} - \mathcal{A}_{s,t} \right \|_{L^p(\Omega)} = 0,
 \end{align*}
 where $|\mathcal{P}_{s,t}| = \sup_{[u,v] \in \mathcal{P}_{s,t}}(v-u)$ denotes the mesh size of $\mathcal{P}_{s,t}$.
\end{Lemma}

Below, we state a consequence of the stochastic sewing lemma that allows us to obtain bounds on the increments of the characteristic function of the occupation measure. The latter is an abstract version of the arguments given in the proofs of \cite{MR4488556, MR4342752}.

\begin{Proposition}\label{prop: sewing local time}
 Fix $T > 0$ and let $(\Omega, \F, (\F_t)_{t \in [0,T]}, \P)$ be a stochastic basis with the usual conditions. Let $X: \Omega \times [0,T] \longrightarrow \R^d$ and $w: \Omega \times [0,T] \longrightarrow \R$ be progressively measurable such that
 \begin{align}\label{eq: rho integral bound}
      \int_s^t \| w_{\tau} \|_{L^p(\Omega)}\, \mathrm{d}\tau \lesssim C(t-s)^{\kappa}
 \end{align}
 holds for some $p \geq 2$ and $\kappa > 1/2$. Define for $f: \R^d \longrightarrow \R$ is continuous and bounded the functional
 \[
     \varphi_t = \int_0^t w_{\tau}f(X_{\tau})\, \mathrm{d}\tau.
 \]
 If there exists $\Gamma > 0$ such that
 \begin{align}\label{eq: conditional expectation}
      \left\| \int_s^t \E\left[ w_{\tau}f(X_{\tau})\ | \ \mathcal{F}_{s}\right]\, \mathrm{d}\tau \right\|_{L^p(\Omega)} \leq \Gamma (t-s)^{\kappa},
 \end{align}
 then there exists $c > 0$ {depending only on $d, p, \kappa$} such that
 \[
      \| \varphi_t - \varphi_s \|_{L^p(\Omega)} \leq c\Gamma (t-s)^{\kappa}, \qquad (s,t) \in \Delta_T^{2}
 \]
\end{Proposition}
\begin{proof}
 Let $\varphi_{s,t} = \varphi_t - \varphi_s$, $(s,t) \in \Delta_T^{2}$ and define $A_{s,t} = \E[\varphi_{s,t}\ | \ \F_s]$. It is clear that $A_{s,s} = 0$ and that $A_{s,t}$ is $\F_t$-measurable. Since, for $(s,r,t) \in \Delta_{T}^3$, we have $\delta_r A_{s,t} = \E[\varphi_{r,t}\ | \ \F_s] - \E[ \varphi_{r,t}\ | \F_r]$, the tower property of conditional expectations yields $\E[\delta_r A_{s,t} \ | \ \F_s] = 0$. Moreover, by assumption \eqref{eq: conditional expectation} we obtain 
 \[
  \| \delta_{r}A_{s,t}\|_{L^p(\Omega)} \leq \|A_{s,t}\|_{L^p(\Omega)} + \|A_{r,t}\|_{L^p(\Omega)} + \|A_{s,r}\|_{L^p(\Omega)} \leq 3\Gamma|t-s|^{\varkappa}.
 \]
 Hence, the assumptions of the stochastic sewing lemma are satisfied for $C_1 = 0$, $\varkappa_1 > 1$ arbitrary, $C_2 = 3\Gamma$ and $\varkappa_2 = \varkappa$. This yields the existence of a process $\mathcal{A}: [0,T] \longrightarrow \R^d$ which satisfies $\mathcal{A}_0 = 0$, $\| \mathcal{A}_{s,t} - A_{s,t}\|_{L^p(\Omega)} \leq cC_2|t-s|^{\varkappa}$, and $\E[ \mathcal{A}_{s,t} - A_{s,t} | \F_s] = 0$. Using the triangle inequality and the $L^p$ estimate on $A_{s,t}$, we find 
 \[
  \| \mathcal{A}_{s,t}\|_{L^p(\Omega)} 
  \leq \| \mathcal{A}_{s,t} - A_{s,t}\|_{L^p(\Omega)} + \|A_{s,t}\|_{L^p(\Omega)} \leq 3c\Gamma |t-s|^{\kappa} + \Gamma|t-s|^{\kappa}.
 \]
 It remains to show that $\mathcal{A}_{s,t} = \varphi_{s,t}$. 
 Note that $\varphi$ is adapted and satisfies $\varphi_{0} = 0$. Moreover, we have $\E[\varphi_{s,t} - A_{s,t}\ | \ \mathcal{F}_s] = 0$, and using \eqref{eq: rho integral bound} we obtain
 \[
  \| \varphi_{s,t} - A_{s,t} \|_{L^p(\Omega)} \leq 2\|\varphi_{s,t}\|_{L^p(\Omega)} \leq 2C |t-s|^{\kappa}.
 \]
 Hence, by the uniqueness in the stochastic sewing lemma, we conclude that $\varphi$ and $\mathcal{A}$ are modifications of each other.
\end{proof}

{Remark that the assertion $\| \varphi_t - \varphi_s \|_{L^p(\Omega)} \leq c\Gamma (t-s)^{\kappa}$ contains the constant $\Gamma$ that stems from \eqref{eq: conditional expectation}. In this way, Proposition \ref{prop: sewing local time} provides a flexible way to deduce a-priori bounds on $\varphi_t - \varphi_s$ in $L^p(\Omega)$ in terms of bounds on its conditional expectation.}

\section{Volterra It\^{o} processes}
\label{sec: VIt\^{o}}

\subsection{Regularity of weighted occupation measures}

In this section, we study the regularity of the local time for the Volterra-It\^{o} process 
\begin{align}\label{eq: VIt\^{o}}
 X_t &= g(t) + \int_0^t K_b(t,s)b_s\, \mathrm{d}s + \int_0^t K_{\sigma}(t,s)\sigma_s\, \mathrm{d}B_s,
\end{align}
where $g: \Omega \times [0,T] \longrightarrow \R^d$ is $\F_0 \otimes \mathcal{B}([0,T])$-measurable, $b: \Omega \times [0,T] \longrightarrow \R^n$ and $\sigma: \Omega \times [0,T] \longrightarrow \R^{k  \times m}$ are progressively measurable, $K_b: \Delta_T^2 \longrightarrow \R^{d \times n}$ and $K_{\sigma}: \Delta_T^2 \longrightarrow \R^{d \times k}$ are measurable such that $\P$-a.s.
\[
 \int_0^t |K_b(t,s)b_s|\, \mathrm{d}s + \int_0^t |K_{\sigma}(t,s)\sigma_s|^2\, \mathrm{d}s < \infty \qquad t \in (0,T].
\]
Then $X$ is well-defined and progressively measurable. Under slight additional conditions, one may also obtain $L^p(\Omega,\P)$-bounds on $X$ and verify that $X-g$ has sample paths in $L_{loc}^q(\R_+)$ for some $q \in [1,\infty]$ or even continuous sample paths. Since we do not need such, we omit the precise conditions, see \cite[Lemma 3.1]{PROMEL2023291} for related arguments.
Let us first introduce the main conditions imposed on the Volterra kernels $K_b, K_{\sigma}$ and coefficients $b,\sigma$:

\begin{enumerate}
    \item[(A1)] {There exists a constant $C_T > 0$} and $\gamma_b, \gamma_{\sigma} \geq 0$ such that for all $(s,t) \in \Delta_T^2$
    \[
     \int_s^t |K_b(t,r)|\, \mathrm{d}r \leq C_T (t-s)^{\gamma_b},
     \ \ \int_s^t |K_{\sigma}(t,r)|^2\, \mathrm{d}r \leq C_T (t-s)^{2\gamma_{\sigma}}
    \]

    \item[(A2)] {There exists a constant $C_T > 0$}, $p \in [2,\infty)$, and $\alpha_b, \alpha_{\sigma} \geq 0$ such that for all $(s,t) \in \Delta_T^2$
    \[
        \|b_t - \E[b_t \ | \F_s] \|_{L^p(\Omega)} \leq C_T(t-s)^{\alpha_b} \ \text{ and } \ \|\sigma_t - \sigma_s\|_{L^p(\Omega)} \leq C_T(t-s)^{\alpha_{\sigma}}.
    \]
    
    \item[(A3)] (Local non-determinism) There exists $H>0$, $C_T > 0$ and {$\rho: \Omega \times [0,T] \longrightarrow [0,1]$} progressively measurable such that for all $(s,t) \in \Delta_T^2$ and $|\xi| = 1$
    \[
        \int_s^t |\sigma_s^{\top}K_{\sigma}(t,r)^{\top}\xi|^2 \mathrm{d}r \geq C_T (t-s)^{2H}\rho_s^2 \ \ \text{ a.s.}
    \]
\end{enumerate}

Condition (A1) is a mild first (respectively second) order moment condition that is used to obtain bounds on the H\"older increments of the process in $L^p(\Omega)$. Since we may always bound $\|b_t - \E[b_t \ | \F_s] \|_{L^p(\Omega)} \leq 2 \|b_t - b_s\|_{L^p(\Omega)}$, condition (A2) can be verified from a mild condition on the increments of the processes $b,\sigma$. However, if $b$ is $\F_0$-measurable (e.g. constant), then $\|b_t - \E[b_t\ | \ \F_s]\|_{L^p(\Omega)} = 0$ and we may pick $\alpha_b > 0$ arbitrarily. Condition (A2) allows to construct a local approximation in the spirit of \cite{MR2668905, FJ2022}, see also \cite{MR3885546} for a general discussion of such approximations and \cite{MR3022725, MR4127334, MR4255174} for applications towards stochastic equations with jumps. Finally, condition (A3) extends the local non-determinism conditions recently used in \cite{MR4488556, MR4342752}. The newly introduced process $(\rho_t)_{t \in [0, T]}$ allows for a flexible treatment of Volterra It\^{o}-diffusions where the diffusion coefficient is not uniformly non-degenerate, as illustrated in the following remark. 

\begin{Remark}\label{remark: nondeterminism}
 If $K_{\sigma}$ satisfies the local non-determinism condition
\begin{align}\label{eq: local nondeterminism K}
 \inf_{t \in (0,T]}\inf_{s \in (0,t]}\inf_{|\xi| = 1}(t-s)^{-2H}\int_s^t |K_{\sigma}(t,r)^{\top}\xi|^2 \mathrm{d}r > 0,
\end{align}
then condition (A3) holds for $\rho_t^2 = 1 \wedge \lambda_{\mathrm{min}}(\sigma_t \sigma_t^{\top})$, where $\lambda_{\min}$ denotes the smallest eigenvalue of a symmetric positive semidefinite matrix.
\end{Remark}

Note that if (A3) holds for $\rho_t$ then it also holds for $\widetilde{\rho}_t$ provided that $\rho_t \geq \widetilde{\rho}_t$. Hence, the restriction $\rho_t \in [0,1]$ is not essential. The following example collects kernels that satisfy \eqref{eq: local nondeterminism K}.

\begin{Example}\label{example: regularizing}
    The following Volterra kernels satisfy condition \eqref{eq: local nondeterminism K}. 
    \begin{enumerate}
        \item[(a)] The Riemann-Liouville fractional kernel is given by 
        \[
        K(t,s) = \frac{(t-s)^{H-\frac{1}{2}}}{\Gamma(H+\frac{1}{2})}\1_{\{s\leq t\}}.
        \]
        with $H > 0$. The process $X_t = \int_0^t K(t,s)\, \mathrm{d}B_s$ is then called Riemann-Liouville fractional Brownian motion.
        
        \item[(b)] The fractional Brownian motion with Hurst index $H \in (0,1)$ is given by $B^H_t = \int_0^t K(t,s)\, \mathrm{d}B_s$ with the kernel
        \[
            K(t,s) = \frac{(t-s)^{H-\frac{1}{2}}}{\Gamma(H + \frac{1}{2})}F_{2,1}\left(H - \frac{1}{2}; H - \frac{1}{2}; H + \frac{1}{2}; 1 - \frac{t}{s}\right) 
        \]
        where $F_{2,1}$ denotes the Hypergeometric function.
        
        \item[(c)] The $\log$-fractional kernel 
        \[
            K(t,s) = \log\left(1 + \frac{1}{t-s}\right)\1_{\{s \leq t\}}
        \]
        satisfies \eqref{eq: local nondeterminism K} for $H = 1/2$.
    \end{enumerate}
\end{Example}

To take into account the possible degeneracy of the diffusion coefficient, we study the \textit{weighted occupation measure} defined in \eqref{eq: weighted measures}. The following is our first main result.

\begin{Theorem}\label{thm: VIto}
 {Suppose that conditions (A1) -- (A3) are satisfied with $p \in [2,\infty)$ given by condition (A2). Suppose that there exists $\chi \in [0,1]$ and a constant $C_T > 0$ such that 
 \begin{align}\label{eq: chi}
   \| \rho_t - \rho_s \|_{L^p(\Omega)} \leq C_T (t-s)^{\chi}, \qquad (s,t) \in \Delta_T^2.
 \end{align}
 Define $\zeta = \min\left\{ \alpha_b + \gamma_b,\ \alpha_{\sigma} + \gamma_{\sigma} \right\}$. Assume $H < \zeta$, there exist $\eta \in (0,1/2)$ and $\delta \in [0, \eta/H]$ such that
 \begin{align}\label{eq: weight function bound}
    \sup_{t \in [0,T]}\E\left[ \rho_t^{p\left( \delta - \frac{\eta}{H}\right)} \right] < \infty,
 \end{align}
  and define the regularity index 
 \begin{align}\label{eq: kappa regularity}
     \kappa_*(\eta) = \frac{1}{\zeta + \eta} \min\left\{ (1\wedge \delta)\chi \left( 1 + \frac{\eta}{H}\right),\ \eta \left(\frac{\zeta}{H} - 1\right) \right\}.
 \end{align}
 Then for each $q \in [1, (1-\eta)p)$, $\e \in (0, 1- \eta - q/p)$, and $\kappa < \kappa_*(\eta) - d/q$ the weighted occupation measure satisfies 
    \[
    \ell^{\delta} \in L^{\frac{p}{q}}(\Omega, \P; C^{1 - \eta - \frac{q}{p} - \e}([0,T]; \mathcal{F}L_q^{\kappa}(\R^d))).
    \] }
\end{Theorem}
\begin{proof} {
\textit{Step 1.} Note that the Fourier transform of $\ell_{s,t}^{\delta} := \ell^{\delta}_t - \ell_s^{\delta}$ is given by 
\begin{align}\label{eq: FT}
  \widehat{\ell}^{\delta}_{s,t}(\xi) = \int_s^t \rho^{\delta}_r\ \mathrm{e}^{\mathrm{i}\langle \xi, X_r\rangle}\, \mathrm{d}r, \qquad \xi \in \R^d,
 \end{align}
where $(s,t) \in \Delta_T^2$. It suffices to prove the existence of a constant $C_T > 0$ such that 
\begin{align}\label{eq: Lp bound FT}
    \left\| \widehat{\ell}_{s,t}^{\delta}(\xi) \right\|_{L^p(\Omega)} \leq C_T (1+|\xi|)^{- \kappa_*(\eta)}(t-s)^{1- \eta}, \qquad \xi \in \R^d.
\end{align}
Indeed, using this bound combined with the Minkowski inequality, we obtain 
 \begin{align*}
  \| \| \ell^{\delta}_{s,t} \|_{\mathcal{F}L_q^{\kappa}} \|_{L^{\frac{p}{q}}(\Omega)}
  &= \left\| \left( \int_{\R^d} (1+|\xi|^2)^{\frac{\kappa q}{2}}|\widehat{\ell}^{\delta}_{s,t}(\xi)|^q\,  \mathrm{d}\xi\right)^{\frac{1}{q}} \right\|_{L^{\frac{p}{q}}(\Omega)}
  \\ &\leq \left( \int_{\R^d} (1+|\xi|^2)^{\frac{\kappa q}{2}} \| \widehat{\ell}^{\delta}_{s,t}(\xi) \|_{L^{\frac{p}{q} q}(\Omega)}^q\, \mathrm{d}\xi \right)^{\frac{1}{q}}
  \\ &\lesssim \left(\int_{\R^d} (1+|\xi|)^{\kappa q - \kappa_*(\eta)q}\, \mathrm{d}\xi \right)^{\frac{1}{q}} (t-s)^{1-\eta}
 \end{align*}
 Noting that the integrals are convergent since $\kappa q - \kappa_*(\eta)q + d < 0$, we obtain
 \[
 \| \| \ell^{\delta}_{s,t} \|_{\mathcal{F}L_q^{\kappa}} \|_{L^{\frac{p}{q}}(\Omega)} \lesssim (t-s)^{1-\eta}.
 \]
 Since $(1-\eta)\frac{p}{q} > 1$ by assumption, the Kolmogorov-Chentsov theorem implies the desired regularity of the weighted occupation measure.  }

{
To verify the bound \eqref{eq: Lp bound FT} for all $\xi \in \R^d$, let us fix $\xi_* > 0$ to be determined later on. When $|\xi| \leq \xi_*$, we trivially obtain $\|\widehat{\ell}^{\delta}_{s,t}(\xi)\|_{L^p(\Omega)} \leq t-s \lesssim (t-s)^{1-\eta}(1+|\xi|)^{-\kappa_*(\eta)}$ since $\rho_t \leq 1$ by assumption. Hence, it suffices to prove \eqref{eq: Lp bound FT} for $|\xi| \geq \xi_*$. For this purpose, we show that its conditional expectation satisfies the bound
\begin{align}\label{eq: 4}
    \left\| \int_s^t \E\left[ \rho_r^{\delta} e^{\mathrm{i}\langle \xi, X_r\rangle} \ | \ \mathcal{F}_s \right]\, \mathrm{d}r \right\|_{L^p(\Omega)} \leq C_T (1+|\xi|)^{- \kappa_*(\eta)}(t-s)^{1- \eta}.
\end{align}
for some constant $C_T > 0$. Then an application of Proposition \ref{prop: sewing local time} for $\varphi_t = \widehat{\ell}^{\delta}_t(\xi)$ readily yields \eqref{eq: Lp bound FT} and hence the assertion. The remaining steps are dedicated to the proof of \eqref{eq: 4}.}

\textit{Step 2.} Let us first construct an approximation of $X_t$ that allows us to extract (locally) the non-deterministic behaviour of the process from the noise. For $t \in (0,T]$, $s \in [0,T]$, and $\e \in (0,1\wedge t)$ let us define the processes
\begin{align*}
    b_s^{\e,t} = \1_{[0,t-\e)}(s)b_s + \1_{[t-\e,t]}(s)\E[b_s \ | \ \F_{t-\e}] \ \text{ and } \ \sigma_s^{\e,t} = \1_{[0,t-\e)}(s)\sigma_s + \1_{[t-\e,t]}(s)\sigma_{t-\e}.
\end{align*}
Using this approximation, let $X_t^{\e,t}$ be given by
\begin{align}\label{eq: local approximation VIt\^{o}}
 X_t^{\e,t} = g(t) + \int_0^t K_b(t,s)b_s^{\e,t}\, \mathrm{d}s + \int_0^t K_{\sigma}(t,s)\sigma_s^{\e,t}\, \mathrm{d}B_s.
\end{align}
Below we show that there exists a constant $C_T > 0$ such that for $t \in (0,T]$ and $\e \in (0,1\wedge t)$ one has  {
\begin{align}\label{prop: local approximation VIto}
    \| X_t - X_t^{\e,t}\|_{L^p(\Omega)} \leq C_T \e^{\zeta}.
\end{align} }
Indeed, by definition of $X_t^{\e,t}$ we obtain
\begin{align*}
     X_t^{\e,t} &= \int_0^{t-\e}K_b(t,s)b_s\, \mathrm{d}s + \int_0^{t-\e}K_{\sigma}(t,s)\sigma_s\, \mathrm{d}B_s
     \\ &\qquad + \int_{t-\e}^t K_b(t,s)b^{\e,t}_{s}\, \mathrm{d}s + \int_{t-\e}^t K_{\sigma}(t,s)\sigma^{\e,t}_{s}\, \mathrm{d}B_s,
\end{align*}
and hence arrive at  
\begin{align*}
        X_t - X_t^{\e,t} &= \int_{t-\e}^t K_b(t,s)(b_s - b^{\e,t}_s)\, \mathrm{d}s + \int_{t-\e}^t K_{\sigma}(t,s)(\sigma_s - \sigma^{\e,t}_s)\, \mathrm{d}B_s.
\end{align*}
By direct computation we find a constant $C_T > 0$ such that {$\|b_s - b_s^{\e,t}\|_{L^p(\Omega)} \leq C_T \e^{\alpha_b}$ holds for all $(s,t) \in \Delta_T^2$}. Hence, using conditions (A1) and (A2) we find
\begin{align*}
    \left\| \int_{t-\e}^t K_b(t,s)(b_s - b^{\e,t}_s)\, \mathrm{d}s\right\|_{L^p(\Omega)}
    &\leq \int_{t-\e}^t |K_b(t,s)|\|b_s - b^{\e,t}_s\|_{L^p(\Omega)}\, \mathrm{d}s
    \\ &\leq C_T \e^{\alpha_b} \int_{t-\e}^t |K_b(t,s)|\, \mathrm{d}s
    \lesssim_T \e^{\alpha_b + \gamma_b},
\end{align*}
{where the constants are independent of $\varepsilon$ and uniform on $[0,T]$.} Likewise, we find another constant $C_T' > 0$ such that $\|\sigma_s - \sigma_s^{\e,t}\|_{L^p(\Omega)} \leq C'_T \e^{\alpha_{\sigma}}$ holds uniformly in $(s,t) \in \Delta_T^2$. Hence, using Jensen's inequality, we obtain
\begin{align*}
        &\ \left\| \int_{t-\e}^t K_{\sigma}(t,s)(\sigma_s - \sigma^{\e,t}_s)\, \mathrm{d}B_s \right\|_{L^p(\Omega)}^p
        \\ &\lesssim \E\left[ \left( \int_{t-\e}^t |K_{\sigma}(t,s)|^2 |\sigma_s - \sigma^{\e,t}_{s}|^2\,  \mathrm{d}s\right)^{\frac{p}{2}} \right]
        \\ &\lesssim \left( \int_{t-\e}^t |K_{\sigma}(t,s)|^2\, \mathrm{d}s\right)^{\frac{p}{2} - 1} \int_{t-\e}^t |K_{\sigma}(t,s)|^2 \| \sigma_s - \sigma^{\e,t}_{s}\|_{L^p(\Omega)}^p\, \mathrm{d}s
        \\ &\leq C_T' \e^{p \alpha_{\sigma}}\left( \int_{t-\e}^t |K_{\sigma}(t,s)|^2\, \mathrm{d}s\right)^{\frac{p}{2}}
        \lesssim_T \e^{p \alpha_{\sigma} + p \gamma_{\sigma}}.
\end{align*}
This proves \eqref{prop: local approximation VIto} and completes step 2.

\textit{Step 3.} Let $a > 0$ be such that $0 < a < 1/H$. Its precise value will be specified in step 4. Furthermore, for $(s,t) \in \Delta_{T}^2$ and $\xi \in \R^d$ let 
\begin{align}\label{eq: u lower bound}
 \e(t,s,\xi) = \frac{t-s}{2|\xi|^a} \ \text{ and } \ \xi_* = \max\left\{1,\ \left(\frac{T}{2}\right)^{\frac{1}{a}}\right\}.
\end{align}
If $|\xi| \geq \xi_*$, then $\e(t,s,\xi) < 1\wedge t$ {due to the particular choice of $\xi_*$,} and hence the approximation $X_t^{\e(t,s,\xi),t}$ from \eqref{eq: local approximation VIt\^{o}} is well-defined with the convention that $X_t^{\e=0,t} = X_t$. For $x \in \R^d$ let $f_{\xi}(x) = \mathrm{e}^{\mathrm{i}\langle \xi, x\rangle}$ and define
\begin{align*}
 \mathbb{A}_{s,t}(\xi) = \int_s^t \E\left[ \rho^{\delta}_{r - \e(r,s,\xi)} f_{\xi}\left(X^{\e(r,s,{\xi}),r}_{r} \right) \ | \ \F_s\right]\, \mathrm{d}r, \qquad |\xi| \geq \xi_*.
\end{align*}
 Let us prove that for $a \in (0, \frac{1}{H})$ and all $(s,t) \in \Delta_{T}^2$ with $s < t$ and $|\xi| \geq \xi_*$, we have 
 \begin{align}\label{lemma: Abb estimate VIt\^{o}}
  \|\mathbb{A}_{s,t}(\xi)\|_{L^p(\Omega)} \lesssim |\xi|^{-\frac{\eta}{H}(1 - Ha)} (t-s)^{1-\eta}.
 \end{align}
 Indeed, let $r \in (s,t]$. The particular form of $\e(r,s,\xi)$ yields $r - \e(r,s,\xi) > s$ and using the definition of the local approximation we obtain $X_r^{\e(r,s,\xi),r} = U_r^{\e(r,s,\xi)} + V_r^{\e(r,s,\xi)}$ where 
 \begin{align}\label{eq: U approximation}
        U_r^{\e(r,s,\xi)} &= g(r) + \int_0^{r-\e(r,s,\xi)} K_b(r,\tau)b_{\tau}\, \mathrm{d}\tau 
        \\ \notag &\qquad + \int_0^{r-\e(r,s,\xi)} K_{\sigma}(r,\tau)\sigma_{\tau}\, \mathrm{d}B_{\tau} + \int_{r-\e(r,s,\xi)}^r K_b(r,\tau)\E[b_{\tau} \ | \ \F_{r - \e(r,s,\xi)}] \, \mathrm{d}\tau
 \end{align}
 is $\F_{r - \e(r,s,\xi)}$-measurable and 
 \begin{align}\label{eq: V approximation}
    V_r^{\e(r,s,\xi)} = \int_{r-\e(r,s,\xi)}^r K_{\sigma}(r,\tau)\sigma_{r-\e(r,s,\xi)}\, \mathrm{d}B_{\tau}
 \end{align}
 is conditionally on $\F_{r-\e(r,s,\xi)}$ {a centered Gaussian random variable} with variance 
 \begin{align*}
     \mathrm{var}(V_r^{\e(r,s,\xi)}\ | \ \F_{r-\e(r,s,\xi)}) = \int_{r-\e(r,s,\xi)}^r K_{\sigma}(r,\tau)\sigma_{r-\e(r,s,\xi)}\left(\sigma_{r-\e(r,s,\xi)}\right)^{\top}K_{\sigma}(r,\tau)^{\top}\, \mathrm{d}\tau.
 \end{align*}
 In particular, for each $\xi \in \R^d$ we obtain from the local non-determinism condition (A3)
 \begin{align} \notag
        \langle \xi, \mathrm{var}(V_r^{\e(r,s,\xi)}|\F_{r-\e(r,s,\xi)}) \xi\rangle &= \int_{r - \e(r,s,\xi)}^r \left|\left(\sigma_{r-\e(r,s,\xi)}\right)^{\top}K_{\sigma}(r,\tau)^{\top}\xi \right|^2 \mathrm{d}\tau
        \\ &\gtrsim \rho_{r - \e(r,s,\xi)}^2 \e(r,s,\xi)^{2H}|\xi|^2 \label{eq: variance bound}
        \\ &= 2^{-2H} \rho_{r - \e(r,s,\xi)}^2 (r-s)^{2H}|\xi|^{2 - 2Ha}, \notag
 \end{align}
 {with a constant that is uniform in $r,s$ by assumption (A3).} Since $r - \e(r,s,\xi) > s$, we may condition on $\F_{r - \e(r,s,\xi)}$ which gives 
 \begin{align*}
        &\ \E[ \rho^{\delta}_{r - \e(r,s,\xi)} f_{\xi}(X_r^{\e(r,s,\xi),r}) \ | \ \F_s ]
        \\ &\qquad \qquad = \E\left[ \rho^{\delta}_{r - \e(r,s,\xi)}\mathrm{e}^{\mathrm{i}\langle \xi, U_r^{\e(r,s,\xi)} \rangle} \mathrm{e}^{-\frac{1}{2}\langle \xi, \mathrm{var}(V_r^{\e(r,s,\xi)}|\F_{r-\e(r,s,\xi)})\xi \rangle} \ | \ \F_s\right]
 \end{align*}
 where we have used the tower property, that $U_{r}^{\e(r,s,\xi)}$ is $\F_{r - \e(r,s,\xi)}$-measurable, and that $V_{r}^{\e(r,s,\xi)}$ is Gaussian conditionally on $\F_{r - \e(r,s,\xi)}$. Using the triangle inequality for conditional expectations, we obtain 
 \begin{align*} \notag
        &\ \left| \E[ \rho^{\delta}_{r - \e(r,s,\xi)}f_{\xi}(X_r^{\e(r,s,\xi),r}) \ | \ \F_s ]\right|
        \\ \notag &\leq \E\left[ \rho^{\delta}_{r - \e(r,s,\xi)} \mathrm{e}^{-\frac{1}{2}\langle \xi, \mathrm{var}(V_r^{\e(r,s,\xi)}|\F_{r-\e(r,s,\xi)}) \xi \rangle} \ | \ \F_s \right]
        \\ \notag &\leq \E\left[ \rho^{\delta}_{r - \e(r,s,\xi)} \exp\left( - 2^{-2H-1} \rho_{r-\e(r,s,\xi)}^2 (r-s)^{2H}|\xi|^{2 - 2Ha}\right) \ | \F_s \right]
        \\ &\lesssim (r-s)^{- \eta}|\xi|^{- \frac{\eta}{H}(1 - Ha)} \E\left[ \rho^{\delta - \frac{\eta}{H}}_{r - \e(r,s,\xi)} \ | \F_s \right]
 \end{align*}
 where we have used $\sup_{x \geq 0}x^{\eta/2H}\mathrm{e}^{-x} < \infty$. Finally, using \eqref{eq: weight function bound} we get
 \[
    \|\mathbb{A}_{s,t}(\xi)\|_{L^p(\Omega)} \lesssim \int_s^t (r-s)^{-\eta} |\xi|^{-\frac{\eta}{H}(1-Ha)} \, \mathrm{d}r
    = \frac{(t-s)^{1 - \eta}}{1-\eta}|\xi|^{-\frac{\eta}{H}(1-Ha)}
 \]
 which yields the desired bound on $\mathbb{A}_{s,t}(\xi)$. 

 \textit{Step 4.} {In this step, we complete the proof of \eqref{eq: 4} which also completes the proof of this theorem. Recall that the Fourier transform of the weighted occupation measure is given in \eqref{eq: FT}. Let $(s,t) \in \Delta_T^2$ with $s < t$, recall that $a \in (0,\frac{1}{H})$ and $\eta \in (0,1/2)$, then using step 3 we obtain for $|\xi| \geq \xi_*$ the estimate }
 \begin{align*}
  \| \E[\widehat{\ell}_{s,t}^{\delta}(\xi)\ | \ \F_s]\|_{L^p(\Omega)}
   &\leq \| \E[\widehat{\ell}_{s,t}^{\delta}(\xi)\ | \ \F_s] - \mathbb{A}_{s,t}(\xi)\|_{L^p(\Omega)} + \| \mathbb{A}_{s,t}(\xi)\|_{L^p(\Omega)}
  \\ &\leq \int_s^t \| \rho^{\delta}_r f_{\xi}(X_{r}) -  \rho^{\delta}_{r - \e(r,s,\xi)}f_{\xi}(X_{r}^{\e(r,s,\xi),r}) \|_{L^p(\Omega)}\, \mathrm{d}r
  \\ &\qquad + |\xi|^{-\frac{\eta}{H}(1-Ha)} (t-s)^{1-\eta}.
 \end{align*}
 The first term on the right-hand side can be estimated by
 \begin{align}\label{eq: 3}
    \notag &\ \int_s^t \| \rho^{\delta}_r f_{\xi}(X_{r}) -  \rho^{\delta}_{r - \e(r,s,\xi)}f_{\xi}(X_{r}^{\e(r,s,\xi),r}) \|_{L^p(\Omega)}\, \mathrm{d}r
    \\ \notag &\leq \int_s^t \| f_{\xi}(X_{r}^{\e(r,s,\xi),r})(\rho^{\delta}_{r} -  \rho^{\delta}_{r - \e(r,s,\xi)}) \|_{L^p(\Omega)}\, \mathrm{d}r 
     + \int_s^t \| \rho^{\delta}_{r}(f_{\xi}(X_{r}) - f_{\xi}(X_{r}^{\e(r,s,\xi),r}) )\|_{L^p(\Omega)}\, \mathrm{d}r 
    \\ &\lesssim \int_s^t \| \rho^{\delta}_r - \rho^{\delta}_{r - \e(r,s,\xi)}\|_{L^p(\Omega)}\, \mathrm{d}r 
    + \int_s^t \|  f_{\xi}(X_{r}) - f_{\xi}(X_{r}^{\e(r,s,\xi),r}) \|_{L^p(\Omega)}\, \mathrm{d}r
    \\ \notag &\lesssim \int_s^t \e(r,s,\xi)^{(1 \wedge \delta)\chi}\, \mathrm{d}r + |\xi| \int_s^t \| X_r - X_r^{\e(r,s,\xi),r}\|_{L^p(\Omega)}\, \mathrm{d}r
    \\ \notag &\lesssim {(t-s)^{1 + (1\wedge \delta)\chi} |\xi|^{- a (1\wedge \delta)\chi } + |\xi|^{1-a\zeta}(t-s)^{1+\zeta}}
    \\ \notag &\lesssim {(t-s)^{1 - \eta} \left( |\xi|^{- a (1\wedge \delta)\chi } + |\xi|^{1-a\zeta}\right)  }
 \end{align}
 where we have using the H\"older inequality, \eqref{eq: chi}, and \eqref{prop: local approximation VIto}. To summarise, we have shown that
 \begin{align*}
     \| \E[\widehat{\ell}_{s,t}^{\delta}(\xi)\ | \ \F_s]\|_{L^p(\Omega)} \lesssim \left( |\xi|^{-a (1\wedge \delta)\chi} + |\xi|^{1 - a \zeta} + |\xi|^{-\frac{\eta}{H}(1 - Ha)} \right)(t-s)^{1-\eta}.
 \end{align*}
 Note that we may still optimise the decay of $|\xi|$ by choosing $a \in (0,1/H)$. Let us take the particular choice $a = \frac{1 +\eta/H}{\zeta + \eta}$, which satisfies $a \in (1/\zeta, 1/H)$ since $\zeta > H$. Then we arrive at
 \begin{align}\label{eq: pointwise bound occupation measure}
     \| \E[\widehat{\ell}_{s,t}^{\delta}(\xi)\ | \ \F_s]\|_{L^p(\Omega)} \lesssim |\xi|^{- \kappa_*(\eta)}(t-s)^{1-\eta} \lesssim (1+|\xi|)^{-\kappa_*(\eta)}(t-s)^{1-\eta}
 \end{align}
 with $\kappa_*(\eta)$ defined in \eqref{eq: kappa regularity}. This proves the desired bound \eqref{eq: 4}.  
\end{proof}

The parameter $\eta$ describes a trade-off between spatial regularity captured by $\kappa < \kappa_*(\eta) - d/q$ and regularity in the time variables. The latter is a well-known effect already present in the case of fractional Brownian motion. The parameter {$\delta$ allows us to track the integrability of $\rho_t^{\delta - \eta/H}$, i.e. suppresses those regions where the noise vanishes. The choice $\delta = \eta/H$ cancels the dependence on the weight, hence corresponds to the usual occupation measure, and removes the dependence on $p$ in \eqref{eq: weight function bound}. In such a case, Theorem \ref{thm: VIto} can be applied to all $p$ that satisfy condition (A2).}

{The parameter $\chi$ is present to capture the loss of regularity caused by the weight $\rho_t$. The next remark shows that, for the usual occupation and self-intersection measure with $\delta = 0$, it plays no role.}

\begin{Remark}\label{remark: delta zero}
    The proofs, particularly \eqref{eq: 3}, show that for $\delta = 0$ no bounds on $\rho_t - \rho_s$ in $L^p(\Omega)$ are required and hence the regularity index takes the form 
    \[
    \kappa_*(\eta) = \frac{\eta}{\zeta + \eta} \left( \frac{\zeta}{H} - 1\right).
    \]
\end{Remark}

The bounds obtained for the weighted occupation measure can also be used to derive regularity for the \textit{weighted self-intersection measure} defined in \eqref{eq: weighted measures}.

\begin{Corollary}\label{cor: selfintersection time}
    Suppose that conditions (A1) -- (A3) are satisfied for $p \in (2,\infty)$. {Suppose that there exists $\chi \in (0,1]$ with property \eqref{eq: chi}, recall that $\zeta$ was defined in Theorem \ref{thm: VIto} and $\kappa_*(\eta)$ is given by \eqref{eq: kappa regularity}. Suppose that $\zeta > H$ and there exists $\eta \in (0,1/2)$ and $\delta \in [0,\eta/H]$ such that \eqref{eq: weight function bound} holds.} Then for each $1 \leq q < (1-\eta)\frac{p}{2}$, $\varepsilon \in (0, 1 - \eta - 2q/p)$, and $\kappa < 2\kappa_*(\eta) - d/q$ the weighted self-intersection measure satisfies 
    \[
    G^{\delta} \in L^{\frac{p}{2q}}(\Omega, \P; C^{1 - \eta - \frac{2q}{p} - \e}([0,T]; \mathcal{F}L_q^{\kappa}(\R^d))).
    \]
\end{Corollary}
\begin{proof}
    Let $(s,t) \in \Delta_T^2$ and $\xi \in \R^d$. Using the definition of the weighted self-intersection measure, we obtain for $\widehat{G}^{\delta}_{s,t}(\xi) = \widehat{G}^{\delta}_t(\xi) - \widehat{G}^{\delta}_s(\xi)$ the representation
\begin{align*}
    \widehat{G}^{\delta}_{s,t}(\xi) &= \int_s^t \int_s^t \rho^{\delta}_r \cdot \rho^{\delta}_{\tau}\cdot \mathrm{e}^{\mathrm{i}\langle \xi, X_r - X_{\tau}\rangle}\, \mathrm{d}\tau \mathrm{d}r
    + \int_0^s \int_s^t \rho^{\delta}_r \cdot \rho^{\delta}_{\tau}\cdot \mathrm{e}^{\mathrm{i}\langle \xi, X_r - X_{\tau}\rangle}\, \mathrm{d}\tau \mathrm{d}r
    \\ &\qquad + \int_s^t \int_0^s \rho^{\delta}_r \cdot \rho^{\delta}_{\tau}\cdot \mathrm{e}^{\mathrm{i}\langle \xi, X_r - X_{\tau}\rangle}\, \mathrm{d}\tau \mathrm{d}r
    \\ &= \widehat{\ell}_{s,t}^{\delta}(\xi)\widehat{\ell}_{s,t}^{\delta}(-\xi) + \widehat{\ell}_{0,s}^{\delta}(\xi)\widehat{\ell}_{s,t}^{\delta}(-\xi) + \widehat{\ell}_{s,t}^{\delta}(\xi)\widehat{\ell}_{0,s}^{\delta}(-\xi).
\end{align*}
 Using \eqref{eq: Lp bound FT} combined with the H\"older inequality, we arrive at 
 \begin{align*}
    \| \widehat{G}^{\delta}_{s,t}(\xi)\|_{L^{p/2}(\Omega)} 
    &\leq \| \widehat{\ell}_{s,t}^{\delta}(\xi)\|_{L^{p}(\Omega)} \|\widehat{\ell}_{s,t}^{\delta}(-\xi)\|_{L^{p}(\Omega)} + \|\widehat{\ell}_{0,s}^{\delta}(\xi)\|_{L^{p}(\Omega)}\|\widehat{\ell}_{s,t}^{\delta}(-\xi)\|_{L^{p}(\Omega)} 
    \\ &\qquad \qquad + \|\widehat{\ell}_{s,t}^{\delta}(\xi)\|_{L^{p}(\Omega)}\|\widehat{\ell}_{0,s}^{\delta}(-\xi)\|_{L^{p}(\Omega)}
    \\ &\lesssim |\xi|^{-2\kappa_*(\eta)}(t-s)^{1-\eta}
 \end{align*}
 for $|\xi| \geq \xi_*$. From this, we deduce, similarly to step 1 in the proof of Theorem \ref{thm: VIto}, the desired regularity in the Fourier-Lebesgue space. 
\end{proof}

\subsection{Regularity of the law}

Let $X$ be a Volterra It\^{o} process and $\ell_t^{\delta}$ the corresponding weighted occupation measure. Taking expectations in the definition of the weighted occupation measure gives $\E[ \ell_t^{\delta}(A)] = \int_0^t \mu_s^{\delta}(A)\, \mathrm{d}s$ where $\mu_s^{\delta}$ denotes the weighted law of the process defined by 
\[
  \mu^{\delta}_t(A) = \E\left[ \rho^{\delta}_t \1_A(X_t) \right], \qquad A \in \mathcal{B}(\R^d).
\]
Thus, if the conditions of Theorem \ref{thm: VIto} are satisfied, then $\ell_t^{\delta} \in \mathcal{F}L_q^{\kappa}(\R^d)$ which gives regularity for the integrated measure $\int_0^t \mu_s^{\delta}\, \mathrm{d}s$. Below we strengthen this observation by proving the absolute continuity with respect to the Lebesgue measure and dimension-independent regularity for its density in the Besov space $B_{1,\infty}^{s}(\R^d)$. For this purpose, we use the following lemma.

{
\begin{Lemma}\label{lemma: DENSITY}\cite[Lemma 2.1]{MR3022725}
    Let $\mu$ be a finite measure on $\R^d$. Assume that there exist $\eta, \kappa \in (0,1)$ and $C > 0$ such that
    \[
        \left| \int_{\R^d} \left( \phi(x+h) - \phi(x)\right)\, \mu(\mathrm{d}x)\right|
        \leq C \| \phi\|_{C_b^{\eta}} |h|^{\eta + \kappa}, \qquad |h| \leq 1
    \]
    holds for all $\phi \in C_b^{\eta}(\R^d)$. Then $\mu$ is absolutely continuous with respect to the Lebesgue measure on $\R^d$. Let $g \in L^1(\R^d)$ be the density of $\mu$. Then there exists another constant $\widetilde{C}_{d,\eta,\kappa} > 0$ that depends on $\eta, \kappa, d$ such that
    \[
        \| g\|_{B_{1,\infty}^{\kappa}} \leq \mu(\R^d) + C \widetilde{C}_{d,\eta,\kappa} < \infty.
    \]
\end{Lemma} }

{
Below we apply this lemma for $\mu^{\delta}_t$ with $t > 0$ fixed. Since $\rho_t$ is bounded by assumption (A3), $\mu_t^{\delta}$ is clearly a finite measure.} The following is our main result on the regularity of $\mu_t^{\delta}$.

\begin{Theorem}\label{thm: VIto density}
 Suppose that conditions (A1) -- (A3) are satisfied. Suppose that there exists $\chi \in (0,1]$ such that 
 \[
    \| \rho_t - \rho_s \|_{L^1(\Omega)} \lesssim (t-s)^{\chi}
 \]
 holds uniformly in $(s,t) \in \Delta_T^2$. If $H < \zeta$, then the following assertions hold:
 \begin{enumerate}
     \item[(a)] {If there exists $\eta > 0$ and $\delta \in [0,\eta/H]$ such that
     \begin{align}\label{eq: 5}
        \sup_{t \in [0,T]}\E\left[ \rho^{\delta - \frac{\eta}{H}} \right] < \infty,
     \end{align} 
     then for each $q \in [1,\infty]$ and $\kappa < \kappa_*(\eta) - d/q$ the weighted law satisfies $\mu^{\delta}_t \in \mathcal{F}L_q^{\kappa}(\R^d)$ and $\| \mu_t^{\delta} \|_{\mathcal{F}L_q^{\kappa}} \lesssim (1\wedge t)^{-\eta}$ for $t \in (0,T]$.
     \item[(b)] If there exists $\delta \in [0,1]$ such that 
     \[
        \sup_{t \in [0,T]}\E\left[ \rho_t^{\delta-1}\right] < \infty,
     \]
     then $\mu^{\delta}_t$ is absolutely continuous with respect to the Lebesgue measure, its density $\mu_t^{\delta}(dx) = g_t^{\delta}(x)dx$ satisfies $g_t^{\delta} \in B_{1,\infty}^{\kappa'}(\R^d)$ for some $\kappa' > 0$ and each $t \in (0,T]$, and it holds that
     \[
        \| g_t^{\delta} \|_{B_{1,\infty}^{\kappa'}} \lesssim (1 \wedge t)^{- H}.
     \] }
 \end{enumerate}
\end{Theorem}
\begin{proof}
    (a) Fix $t \in (0,T]$ and let $\xi \in \R^d$ with $|\xi| \geq 1$. Let $\kappa < \kappa_*(\eta) - d/q$, define $\e = (1\wedge t)|\xi|^{-a}$ with $a = \frac{1 +\eta/H}{\zeta + \eta}$, and let $X_t^{\e,t}$ be given by \eqref{eq: local approximation VIt\^{o}}. Then we obtain
    \begin{align*}
        \left| \E\left[ \rho^{\delta}_{t}\cdot \mathrm{e}^{\mathrm{i}\langle \xi, X_t\rangle} \right]\right| 
        &\leq \left| \E\left[ \rho^{\delta}_t \cdot \mathrm{e}^{\mathrm{i}\langle \xi, X_t\rangle} \right] - \E\left[ \rho^{\delta}_{t-\e}\cdot \mathrm{e}^{\mathrm{i}\langle \xi, X^{\e,t}_t\rangle} \right] \right| + \left| \E\left[ \rho^{\delta}_{t-\e}\cdot \mathrm{e}^{\mathrm{i}\langle \xi, X^{\e,t}_t\rangle} \right]\right|.
    \end{align*}
    The first term satisfies
    \begin{align*}
        &\ \left| \E\left[ \rho^{\delta}_t \cdot \mathrm{e}^{\mathrm{i}\langle \xi, X_t\rangle} \right] - \E\left[ \rho^{\delta}_{t-\e}\cdot \mathrm{e}^{\mathrm{i}\langle \xi, X^{\e,t}_t\rangle} \right] \right| 
        \\ &\leq \left| \E\left[ (\rho^{\delta}_t - \rho^{\delta}_{t-\e})\cdot\mathrm{e}^{\mathrm{i}\langle \xi, X^{\varepsilon,t}_{t}\rangle}\right]\right| 
        + \left| \E\left[ \rho^{\delta}_{t}\cdot (\mathrm{e}^{\mathrm{i}\langle \xi, X_t\rangle} - \mathrm{e}^{\mathrm{i}\langle \xi, X_t^{\e,t}\rangle})\right] \right|
        \\ &\lesssim \| \rho_t - \rho_{t-\e}\|_{L^1(\Omega)}^{1\wedge \delta} + |\xi| \| X_t - X_t^{\e,t} \|_{L^p(\Omega)} \lesssim |\xi|^{-a (1\wedge \delta)\chi} + |\xi|^{1 - a\zeta}
    \end{align*}
    where we have used \eqref{prop: local approximation VIto} and $\| \rho_t - \rho_s \|_{L^1(\Omega)} \lesssim (t-s)^{\chi}$. For the second term, we use the decomposition $X_t^{\e,t} = U_t^{\e} + V_t^{\e}$ given as in \eqref{eq: U approximation} and \eqref{eq: V approximation}, that $V_t^{\e}$ is conditionally on $\F_{t-\e}$ Gaussian and the lower bound on the variance \eqref{eq: variance bound} to find that 
    \begin{align*}
        \left| \E\left[ \rho^{\delta}_{t-\e}\cdot \mathrm{e}^{\mathrm{i}\langle \xi, X^{\e,t}_t\rangle} \right]\right| &= \left|\E\left[ \rho^{\delta}_{t-\e}\cdot \mathrm{e}^{\mathrm{i}\langle \xi, U_t^{\e}\rangle} \E\left[ \mathrm{e}^{ - \frac{1}{2}\langle\xi, \mathrm{var}(V_t^{\e}|\F_{t-\e})\xi\rangle}\ | \ \F_{t-\e}\right]\right] \right|
        \\ &\lesssim \E\left[ \rho^{\delta}_{t - \e} \exp\left( - \frac{1}{2} \rho_{t-\e}^2 (1\wedge t)^{2H}|\xi|^{2 - 2Ha}\right) \right]
        \\ &\lesssim \E\left[ \rho^{\delta}_{t - \e} \rho_{t-\e}^{-\eta/H} \right] \cdot (1\wedge t)^{-\eta} |\xi|^{-\frac{\eta}{H}(1-Ha)}
        \lesssim (1\wedge t)^{-\eta} |\xi|^{-\frac{\eta}{H}(1-Ha)}
    \end{align*}
    where we have used $\sup_{x \geq 0}x^{\eta/2H}e^{-x} < \infty$ {and condition \eqref{eq: 5}}. Combining both estimates and using the particular form of $a$, we arrive at
    \begin{align*}
         \left| \E\left[ \rho^{\delta}_{t}\cdot \mathrm{e}^{\mathrm{i}\langle \xi, X_t\rangle} \right]\right| &\lesssim (1\wedge t)^{-\eta} |\xi|^{-\kappa_*(\eta)}.
    \end{align*}
    The latter readily implies $\mu_t^{\delta} \in \mathcal{F}L_{q}^{\kappa}(\R^d)$ and the desired bound in the $\mathcal{F}L_q^{\kappa}(\R^d)$ norm.

    (b) Let $t > 0$, $\e = (1\wedge t)|h|^{a}$ with $|h| \leq 1$, and $a > 0$ to be fixed later on. Let $\Delta_h \phi(x) = \phi(x+h) - \phi(x)$ denote the difference operator. Fix $\phi \in C_b^{\eta}(\R^d)$ with $\eta \in (0,1)$. Then we obtain 
    \begin{align} \notag
    \left| \E\left[ \rho^{\delta}_t\cdot \Delta_h \phi(X_t)\right] \right|
    &\leq \left| \E\left[ (\rho^{\delta}_t - \rho^{\delta}_{t-\e})\cdot \Delta_h \phi(X^{\varepsilon,t}_t)\right] \right| + \left| \E\left[ \rho^{\delta}_{t}\cdot (\Delta_h \phi(X_t) - \Delta_h \phi(X_{t}^{\e,t}))\right] \right|
    \\ &\qquad \qquad + \left| \E\left[ \rho^{\delta}_{t-\e}\cdot \Delta_h \phi(X_t^{\e,t})\right] \right| \notag
    \\ &\lesssim \| \phi\|_{C_b^{\eta}}|h|^{\eta}\| \rho^{\delta}_t - \rho^{\delta}_{t-\e}\|_{L^1(\Omega)}
    + \|\phi\|_{C_b^{\eta}}\|X_t - X_t^{\e,t}\|_{L^1(\Omega)}^{\eta} \notag
     \\ &\qquad \qquad + \left| \E\left[ \rho^{\delta}_{t-\e}\cdot \Delta_h \phi(X_t^{\e,t})\right] \right| \notag
    \\ &\lesssim  \| \phi\|_{C_b^{\eta}}|h|^{\eta} \e^{(1\wedge \delta)\chi} + \| \phi\|_{C_b^{\eta}}\e^{\zeta \eta} + \left| \E\left[ \rho^{\delta}_{t-\e}\cdot \Delta_h \phi(X_t^{\e,t})\right] \right|. \label{eq: 6}
    \end{align} {
    where we have used $\rho_t \leq 1$, $\| \rho_t - \rho_s \|_{L^1(\Omega)} \lesssim (t-s)^{\chi}$, \eqref{prop: local approximation VIto}, and}
    \begin{align*}
        \left| \E\left[ \rho^{\delta}_{t}\cdot (\Delta_h \phi(X_t) - \Delta_h \phi(X_{t}^{\e,t}))\right] \right| 
        &\leq \|\Delta_h \phi(X_t) - \Delta_h \phi(X_t^{\e,t})\|_{L^1(\Omega)}
        \\ &\leq \|\phi\|_{C_b^{\eta}}\|X_t - X_t^{\e,t}\|_{L^1(\Omega)}^{\eta}.
    \end{align*}

    {To estimate the last term in \eqref{eq: 6}}, write $X_t^{\e,t} = U_t^{\e} + V_t^{\e}$ as before. Recall that $V_t^{\e}$ is conditionally on $\F_{t-\e}$-Gaussian with covariance \eqref{eq: variance bound}. Let $f_t^{\e}(X_{t-\e},z)\mathrm{d}z$ be the law of $V_t^{\e}$ and $\lambda_{\mathrm{min}}( \mathrm{var}(V_t^{\e}|\F_{t-\e}))$ denote the smallest eigenvalue of its conditional variance. Then 
 \begin{align*}
  \| \Delta_{-h} f_t^{\e}(X_{t-\e},\cdot) \|_{L^1(\R^d)}
  &\lesssim |h|\| \partial_z f_t^{\e}(X_{t-\e},\cdot)\|_{L^1(\R^d)} 
  \\ &\lesssim |h| \left( \lambda_{\mathrm{min}}( \mathrm{var}(V_t^{\e}|\F_{t-\e}))\right)^{-1/2}
  \\ &\lesssim |h|\rho_{t-\e}^{-1}\e^{-H}
  \\ &= (1\wedge t)^{-H}|h|^{1-Ha} \rho_{t-\e}^{-1}
 \end{align*}
 where the last inequality is a consequence of \eqref{eq: variance bound} and the particular form of the density $f_t^{\e}$. Since $U_t^{\e}$ is $\F_{t-\e}$-measurable, conditioning on $\F_{t-\e}$ finally gives 
 \begin{align*}
     \E\left[ \rho^{\delta}_{t-\e}\cdot \Delta_h \phi(X_t^{\e,t})\right] 
     &= \E\left[ \rho^{\delta}_{t-\e} \cdot \E\left[\Delta_h \phi(U_t^{\e} + V_t^{\e}) \ | \ \F_{t-\e} \right]\right]
     \\ &= \E\left[ \rho^{\delta}_{t-\e} \cdot \int_{\R^d} \Delta_h \phi(U_t^{\e} + z) f_t^{\e}(X_{t-\e},z) \mathrm{d}z \right]
     \\ &= \E\left[ \rho^{\delta}_{t-\e} \cdot \int_{\R^d}  \phi(U_t^{\e} + z) \Delta_{-h} f_t^{\e}(X_{t-\e},z) \mathrm{d}z \right].
 \end{align*}
 Thus, in view of the $L^1(\R^d)$ estimate on $f_t^{\e}$, we arrive at 
 \begin{align*}
   \left| \E[ \rho^{\delta}_{t-\e} \cdot \Delta_h\phi(X^{\e,t}_{t}) ] \right| 
  &\leq \| \phi\|_{C_b^{\eta}} \E\left[ \rho^{\delta}_{t-\e}\cdot \| \Delta_{-h} f_t^{\e}(X_{t-\e},\cdot) \|_{L^1(\R^d)} \right] 
  \\ &\lesssim \| \phi\|_{C_b^{\eta}}|h|^{1- Ha}(1\wedge t)^{-H}
 \end{align*}
 where we have used $\sup_{t \in [0,T]}\E[ \rho_t^{\delta-1}] < \infty$. Using the particular form of $\e$ and combining the above estimates, we obtain
 \begin{align*}
     \left| \E\left[ \rho^{\delta}_t\cdot \Delta_h \phi(X_t)\right] \right| 
     &\lesssim \| \phi\|_{C_b^{\eta}}(1\wedge t)^{-H} |h|^{\eta + \kappa(a,\eta)}
 \end{align*}
 where $\kappa(a,\eta) = \min\{a \chi,\ (\zeta a - 1)\eta,\ 1-Ha - \eta \}$. Since $\zeta > H$ we find $a \in (0,1/H)$ such that $\zeta a - 1 > 0$. Hence, letting $\eta \in (0,1)$ be small enough gives $1 - Ha - \eta > 0$ and thus $\kappa(a,\eta) > 0$. {The assertion follows from Lemma \ref{lemma: DENSITY}. }
\end{proof}

By using an anisotropic version of \cite[Lemma 2.1]{MR3022725} as given in \cite{MR4255174}, one may also obtain refined regularity in anisotropic Besov spaces for Volterra It\^{o} processes driven where $\sigma_t$ is diagonal, $K_{\sigma} = \mathrm{diag}(K_1,\dots, K_d)$, and each $K_j$ satisfies (A3) with some $H_j > 0$. The latter covers, in particular, the case where the noise is given by $(B_t^{H_1}, \dots, B_t^{H_d})$ where $B^{H_1},\dots, B^{H_d}$ are independent fractional Brownian motions. Details of such a result are given in \cite{FGW25} and provide an important tool to prove the failure of the Markov property for stochastic Volterra processes.

\subsection{Examples}

In this section, we illustrate the results by two examples that complement the existing literature. Consider a Volterra It\^{o} process \eqref{eq: VIt\^{o}} where $\sigma_t \equiv \sigma \in \R^{k \times m}$ is (for simplicity) deterministic such that $\mathrm{det}(\sigma \sigma^{\top}) > 0$, i.e.
\begin{align}\label{eq: 1}
    X_t = g(t) + \int_0^t K_b(t,s)b_s\, \mathrm{d}s + \int_0^t K_{\sigma}(t,s)\sigma \, \mathrm{d}B_s.
\end{align}
Here $g: \Omega \times [0,T] \longrightarrow \R^d$ is $\mathcal{F}_0\ \otimes \ \mathcal{B}([0,T])$-measurable, $B$ is an $m$-dimensional Brownian motion, and $b: \Omega \times [0,T] \longrightarrow \R^n$ is progressively measurable. {Below we study the space-time H\"older regularity for the densities of the occupation measure $\ell_t$ and the self-intersection measure $G_t$ defined in \eqref{eq: no weight}.}

When $b_s \equiv 0$, the process is Gaussian and our results obtained in Theorem \ref{thm: VIto} essentially coincide with those obtained in \cite{MR4488556, MR4342752}. More recently, in \cite[Theorem 2.16]{BLM23}, the authors studied the regularity of the local time for the perturbation $X_t = \psi_t + B_t^H$ of the fractional Brownian motion $B^H$ with an adapted path $\psi$ of finite $1$-variation. While their result applies to a general class of perturbations without structural assumptions, below we prove a similar result for \eqref{eq: 1}, which corresponds to the class of perturbations $\psi_t = g(t) + \int_0^t K_b(t,s)b_s\, \mathrm{d}s$ that do not need to have finite $1$-variation. 

\begin{Example}
  Suppose that there exist $\gamma \geq 0$, $C_T > 0$ such that for all $(s,t) \in \Delta_T^2$
  \[
    \int_s^t |K_b(t,r)|\, \mathrm{d}r \leq C_T(t-s)^{\gamma}\ \text{ and } \ 
    \int_s^t |K_{\sigma}(t,r)|^2\, \mathrm{d}r < \infty,
  \]
  and that $K_{\sigma}$ satisfies \eqref{eq: local nondeterminism K} with constant $H > 0$. Moreover, assume that there exists $\alpha \geq 0$ and $p \in [2,\infty)$ such that for all $(s,t) \in \Delta_T^2$
 \[
    \|b_t - \E[b_t \ | \ \F_s]\|_{L^p(\Omega)} \lesssim (t-s)^{\alpha}.
 \]
 Finally assume that $\mathrm{det}(\sigma \sigma^{\top}) > 0$ and that $\alpha + \gamma > 0$. Define 
  \[
  \kappa_*(\eta) = \frac{\alpha + \gamma}{\alpha + \gamma + \eta}\frac{\eta}{H} - \frac{\eta}{\alpha + \gamma + \eta}, \qquad \eta \in (0,1/2).
  \]  
  Then the local time $\ell_t$ satisfies a.s.
  \[
    \ell \in C^{1 - \frac{1}{p} - \eta - \e}([0,T];\mathcal{C}^{\kappa_*(\eta) - d - \e}(\R^d)), \qquad \eta \in (0,1/2), \ \e > 0.
  \]
  In particular, if $H\left(2d + \frac{d+1}{\alpha + \gamma}\right) < 1$ then $\ell$ has a jointly H\"older continuous density. Moreover, if $p > 2$, then the self-intersection time $G_t$ satisfies
  \[
  G \in C^{1 - \frac{2}{p} - \eta - \e}([0,T]; \mathcal{C}^{2 \kappa_*(\eta) - d - \e}(\R^d)), \qquad \eta \in (0,1/2), \ \e > 0.
  \]
  In particular, if $H\left(d + \frac{\frac{d}{2} + 1}{\alpha + \gamma}\right) < 1$, then $G$ has a jointly H\"older continuous density.
\end{Example}
\begin{proof}
 Condition (A1) is satisfied for $\gamma_b = \gamma$ and $\gamma_{\sigma} = 0$. Condition (A2) holds for $\alpha_b = \alpha$ and $\alpha_{\sigma} > 0$ arbitrary, while condition (A3) holds due to assumption \eqref{eq: local nondeterminism K} with $\rho_t^2 = \lambda_{\min}(\sigma \sigma^{\top}) > 0$. Thus {let us apply Theorem \ref{thm: VIto} with $\eta \in (0,1/2)$, where condition \eqref{eq: chi} holds for any choice of $\chi > 0$, and also \eqref{eq: weight function bound} is satisfied for $\delta = 0$ since $\rho_t$ is deterministic and constant.} Thus we obtain $\zeta = \alpha + \gamma > 0$ and hence $\kappa_*(\eta)$ is given by \eqref{eq: kappa regularity} with $\eta \in (0,1/2)$. An application of Theorem \ref{thm: VIto} gives {$\ell \in L^{\frac{p}{q}}(\Omega, \P; C^{1 - \eta - \frac{q}{p} - \e}([0,T]; \mathcal{F}L_q^{\kappa}(\R^d)))$} for $\e \in (0, 1 - \eta - q/p)$, {$q \in [1, (1-\eta)p)$}, and $\kappa < \kappa_*(\eta) - d/q$. An application of Lemma \ref{lemma: Fourier Lebesgue embedding} with $q = 1$ gives the desired a.s. regularity of $\ell$. In particular, if $H\left(2d + \frac{d+1}{\alpha + \gamma}\right) < 1$, then we find $\eta \in (0,1/2)$ close enough to $1/2$ and $\e > 0$ small enough such that $\kappa_*(\eta) - d - \e > 0$ and $1 - \frac{1}{p} - \eta - \e \geq \frac{1}{2} - \eta - \e > 0$ which shows that $\ell$ is jointly H\"older continuous. For the self-intersection time, we apply Corollary \ref{cor: selfintersection time} instead and argue similarly.  
\end{proof}

Kernels $K_b, K_{\sigma}$ that satisfy the conditions above are given in Example \ref{example: regularizing}. In our second example we consider a particular choice of $K_{\sigma}$ for which $\int_0^t K_{\sigma}(t,s)\sigma \mathrm{d}B_s$ is a multi-dimensional $q$-$\log$ Brownian motion (see \cite{MR2132395}). The latter provides an example of a $C^{\infty}$-regularising Gaussian process as recently studied in \cite{MR4342752}. Processes that are $C^{\infty}$-regularising are characterised by the property that their local time is a.s. an element of the test function space $\mathcal{D}(\R^d)$. In our second example, we construct a general class of Volterra It\^{o} processes that are $C^{\infty}$-regularising.

\begin{Example}
    In the notation of \eqref{eq: 1}, suppose that $k = d$. Fix $T \in (0,1)$, $q > 1/2$ and let 
    \[
        K_{\sigma}(t,s) = \frac{1}{|(t-s)\log(1/(t-s))^{2q}|^{1/2}} \1_{\{s \leq t\}}, \qquad (s,t) \in \Delta_T^2.
    \]
    Suppose that $\sup_{t \in [0,T]}\|b_t\|_{L^p(\Omega)} < \infty$ for some $p \in [2,\infty)$, and there exists $\gamma > 1/p$ such that  
    \[
        \int_s^t |K_b(t,r)|\, \mathrm{d}r + \int_0^s |K_b(t,r) - K_b(s,r)|\, \mathrm{d}r \leq C_T (t-s)^{\gamma}.
    \]
    {holds for all $(s,t) \in \Delta_T^2$ and some constant $C_T > 0$.} Finally, assume that $\mathrm{det}(\sigma \sigma^{\top}) > 0$. If $g$ is a.s. continuous, then $X$ given by \eqref{eq: 1} is $C^{\infty}$-regularising, i.e., its local time belongs a.s. to the test function space $\mathcal{D}(\R^d)$. Moreover, if $p > 2$, then also the self-intersection time belongs to $\mathcal{D}(\R^d)$.
\end{Example}
\begin{proof}
    As in previous example, condition (A1) is satisfied for $\gamma_b = \gamma$ and $\gamma_{\sigma} = 0$, while condition (A2) holds for $\alpha_b = 0$ and $\alpha_{\sigma} > 0$ arbitrary. Finally, since $\int_s^t K_{\sigma}(t,r)^2\, \mathrm{d}r = \frac{1}{2q - 1}\log(1/(t-s))^{1-2q}$, we see that \eqref{eq: local nondeterminism K} holds for any choice for $H > 0$ with $\rho_t^2 = \lambda_{\min}(\sigma \sigma^{\top}) > 0$. In particular, we may choose $\delta = 0$, which corresponds to the ordinary occupation measure $\ell_t$, and self-intersection measure $G_t$, see \eqref{eq: no weight}. Note that \eqref{eq: chi} holds for any choice of $\chi > 0$ and hence $\zeta = \gamma > 0$. In particular, for each $\eta \in (0,1/2)$ and each $\kappa > 0$ we find $H > 0$ sufficiently small such that $\kappa_*(\eta) > \kappa$. Thus by Theorem \ref{thm: VIto} we obtain $\ell \in C^{1 - \frac{1}{p} - \e}([0,T]; \mathcal{C}^{\kappa}(\R^d))$ a.s. for any $\e, \kappa > 0$, where we have absorbed $\eta$ into $\e$. Since $\gamma p > 1$ and $\sup_{t \in [0,T]}\|b_t\|_{L^p(\Omega)} < \infty$, an application of the Kolmogorov-Chentsov theorem shows that $t \longmapsto \int_0^t K_b(t,s)b_s\, \mathrm{d}s$ is a.s. continuous. Finally, since also $t \longmapsto \int_0^t K_{\sigma}(t,s)\sigma \mathrm{d}B_s$ is continuous (see \cite{MR2132395} Definition 18 and below), also $X$ given by \eqref{eq: 1} has continuous sample paths. Thus $\ell$ has a.s. compact support which proves the assertion. If $p > 2$, then we may apply Corollary \ref{cor: selfintersection time} and argue in the same way. 
\end{proof}

In both examples we have assumed that $\mathrm{det}(\sigma \sigma^{\top}) > 0$ so that (A3) reduces to \eqref{eq: local nondeterminism K}. Clearly, for deterministic $\sigma$ it does not {make} sense to consider the case $\mathrm{det}(\sigma \sigma^{\top}) = 0$ since by Remark \ref{remark: nondeterminism} one has $\rho_t^2 = \lambda_{\min}(\sigma \sigma^{\top}) = 0$ and hence condition (A3) becomes trivial. For non-deterministic $\rho_t$, however, cases where $\rho_t$ may vanish in certain regions of the state-space are covered and will be discussed in the next section for solutions of \eqref{eq: CVSDE}.

\section{Stochastic Volterra equations} 

\subsection{Regularity}

In this section, we consider the case where the Volterra It\^{o}-process is given by a stochastic Volterra process obtained from a continuous solution of \eqref{eq: CVSDE}. Following \cite{PROMEL2023291}, a weak $L^p$-solution with $p \in [1,\infty)$ consists of a filtered probability space with the usual conditions, an $(\F_t)_{t \in [0,T]}$-Brownian motion, and a progressively measurable process $X \in L^p(\Omega \times [0,T])$ on $\R^d$ such that 
\begin{align}\label{eq: integrability condition}
 \int_0^t |K_b(t,s)b(X_s)|\, \mathrm{d}s + \int_0^t |K_{\sigma}(t,s)\sigma(X_s)|^2\, \mathrm{d}s < \infty, \qquad t \in (0,T]
\end{align}
and \eqref{eq: CVSDE} holds $\P \otimes dt$-a.e. We say that $X$ is a continuous weak solution, if \eqref{eq: integrability condition} holds, $X$ is adapted with a.s. continuous sample paths, and \eqref{eq: CVSDE} holds a.s.. A weak solution is called strong if it is adapted to the filtration generated by the Brownian motion. 

Below we introduce our assumptions on the Volterra kernels and coefficients $b,\sigma$. For a Volterra kernel $K$ and $p \in [1,\infty)$, define 
\[
 \omega_p(t,s; K) = \left(\int_0^s \left| K(t,r) - K(s,r)\right|^p\, \mathrm{d}r \right)^{\frac{1}{p}}
        + \left(\int_s^t |K(t,r)|^p\, \mathrm{d}r\right)^{\frac{1}{p}}.
\]
Suppose that the Volterra kernels $K_b, K_{\sigma}$ satisfy the condition: 
\begin{enumerate}
 \item[(B1)] There exists $\e > 0$ such that $K_b, K_{\sigma}$ satisfy
 \[
  \sup_{t \in [0,T]}\left(\int_0^t |K_b(t,r)|^{1+\e}\, \mathrm{d}r + \int_0^t |K_{\sigma}(t,r)|^{2+\e}\, \mathrm{d}r\right) < \infty.
 \]
 Moreover, there exist $\gamma_b, \gamma_{\sigma} > 0$ and $C_T>0$ such that for all $(s,t) \in [0,T]$ with $s < t$
 \[
  \omega_1(t,s;K_b) \leq C_T (t-s)^{\gamma_b} \ \ \text{ and } \ \ 
  \omega_2(t,s;K_{\sigma}) \leq C_T (t-s)^{\gamma_{\sigma}}.
 \]

 \item[(B2)] (local-nondeterminism) There exist constants $C_*, H > 0$ and $0 \leq \sigma_* \in C^{\theta}(\R^d)$ with $\theta \in (0,1]$ and $\sigma_* \in [0,1]$ such that for all $x \in \R^d$, $t \in (0,T]$, $s \in (0,t]$, and $|\xi| = 1$
 \[
     \int_{s}^{t} | \sigma(x)^{\top}K_{\sigma}(t,r)^{\top}\xi|^2 \mathrm{d}r \geq C_* (t-s)^{2H}\sigma_*(x)^2.
 \]
\end{enumerate}

In condition (B1), the behaviour of $\omega_1, \omega_2$ on the diagonal $s = t$ determines the sample path regularity of $X$, see Proposition \ref{prop: CVSDE hoelder} and \cite[Lemma 3.1]{PROMEL2023291}. Condition (B2) is a variant of local non-determinism conditions now formulated for general stochastic Volterra processes. The term $\sigma_*$ allows for a flexible treatment of processes where the diffusion coefficient is not uniformly non-degenerate. However, if $\sigma$ is uniformly non-degenerate, then we may take $\sigma_* = 1$ and condition (B2) reduces to the local non-determinism condition \eqref{eq: local nondeterminism K}. As before, without loss of generality, we may suppose that $\sigma_* \in [0,1]$.

\begin{Proposition}\label{prop: CVSDE hoelder}
 Suppose that $b,\sigma$ are measurable with linear growth and that $K_b, K_{\sigma}$ satisfy condition (B1). Let $X$ be a continuous weak solution of \eqref{eq: CVSDE}. Then for each $p \in [2,\infty)$ with $\sup_{t \in [0,T]} \|g(t)\|_{L^p(\Omega)}< \infty$, it holds that
 \[
  \sup_{t \in [0,T]}\|X_t\|_{L^p(\Omega)} \lesssim 1 + \sup_{t \in [0,T]}\| g(t)\|_{L^p(\Omega)}.
 \]
 Moreover, for all $(s,t) \in \Delta_T^2$ 
 \[
  \| X_t - X_s \|_{L^p(\Omega)} \lesssim \| g(t) - g(s)\|_{L^p(\Omega)} + (t-s)^{\gamma_b \wedge \gamma_{\sigma}}.
 \]
 In particular, for each $\gamma \in \left(0, \gamma_b \wedge \gamma_{\sigma} - \frac{1}{p}\right)$, the process $X - g$ has a modification that belongs to $L^p(\Omega; C^{\gamma}([0,T]))$.
\end{Proposition}
\begin{proof}
The first inequality follows essentially from the proof of \cite[Lemma 3.4]{PROMEL2023291}, {see also \cite{MR2565842}.} For the second assertion write for $0 \leq s \leq t \leq T$
    \begin{align*}
        X_t - X_s &= g(t) - g(s) + \int_0^s \left( K_b(t,r) - K_b(s,r)\right)b(X_r)\, \mathrm{d}r
        + \int_s^t K_b(t,r)b(X_r)\, \mathrm{d}r
        \\ &\qquad + \int_0^s \left( K_{\sigma}(t,r) - K_{\sigma}(s,r)\right)\sigma(X_r)\, \mathrm{d}B_r
        + \int_s^t K_{\sigma}(t,r)\sigma(X_r)\, \mathrm{d}B_r.
    \end{align*}
    Hence we obtain
    \begin{align*}
        &\ \left\| \int_0^s \left( K_b(t,r) - K_b(s,r)\right)b(X_r)\, \mathrm{d}r
        + \int_s^t K_b(t,r)b(X_r)\, \mathrm{d}r \right\|_{L^p(\Omega)}
        \\ &\leq \int_0^s \left| K_b(t,r) - K_b(s,r)\right| \|b(X_r)\|_{L^p(\Omega)}\, \mathrm{d}r
        + \int_s^t |K_b(t,r)|\|b(X_r)\|_{L^p(\Omega)}\, \mathrm{d}r
        \\ &\lesssim \int_0^s \left| K_b(t,r) - K_b(s,r)\right|\,  \mathrm{d}r
        + \int_s^t |K_b(t,r)|\, \mathrm{d}r
        \\ &\lesssim (t-s)^{\gamma_b}
    \end{align*}
    where we have used the $L^p$-bound on $X$ combined with the linear growth of $b$. For the remaining two terms, we obtain from Jensen's inequality
    \begin{align*}
        &\ \left\| \int_0^s \left( K_{\sigma}(t,r) - K_{\sigma}(s,r)\right)\sigma(X_r)\, \mathrm{d}B_r
        + \int_s^t K_{\sigma}(t,r)\sigma(X_r)\, \mathrm{d}B_r\right\|_{L^p(\Omega)}^p
        \\ &\lesssim \E\left[ \left( \int_0^s \left| K_{\sigma}(t,r) - K_{\sigma}(s,r)\right|^2|\sigma(X_r)|^2\, \mathrm{d}r \right)^{\frac{p}{2}} \right]
        + \E\left[ \left( \int_s^t |K_{\sigma}(t,r)|^2|\sigma(X_r)|^2\, \mathrm{d}r\right)^{\frac{p}{2}} \right]
        \\ &\leq \left( \int_0^s \left| K_{\sigma}(t,r) - K_{\sigma}(s,r)\right|^2\, \mathrm{d}r\right)^{\frac{p}{2} - 1} \int_0^s \left| K_{\sigma}(t,r) - K_{\sigma}(s,r)\right|^2\E\left[|\sigma(X_r)|^p\right]\, \mathrm{d}r
        \\ &\qquad + \left( \int_s^t |K_{\sigma}(t,r)|^2\, \mathrm{d}r\right)^{\frac{p}{2} - 1}  \int_s^t |K_{\sigma}(t,r)|^2\E[|\sigma(X_r)|^p]\, \mathrm{d}r 
        \\ &\lesssim \left( \int_0^s \left| K_{\sigma}(t,r) - K_{\sigma}(s,r)\right|^2\, \mathrm{d}r\right)^{\frac{p}{2}} + \left( \int_s^t |K_{\sigma}(t,r)|^2\, \mathrm{d}r\right)^{\frac{p}{2}}
        \\ &\lesssim (t-s)^{p\gamma_{\sigma}}.
    \end{align*}
    This proves the desired $L^p$-bound. The H\"older continuity is a consequence of the Kolmogorov-Chentsov theorem.
\end{proof}

The next remark addresses an improvement of condition (B1) with $\varepsilon = 0$.

\begin{Remark}\label{remark: convolution}
    {It follows from \cite[Theorem 3.1]{MR2565842} and the definition of $\mathcal{K}_0$ therein, that condition (B1) be weakened to $\varepsilon = 0$ provided that  
    \[
        \limsup_{\eta \searrow 0}\sup_{t \in [0,T]} \int_t^{t+\eta}|K_b(t+\eta,r)|\, \mathrm{d}r + \limsup_{\eta \searrow 0}\sup_{t \in [0,T]} \int_t^{t+\eta}|K_{\sigma}(t+\eta,r)|^2\, \mathrm{d}r = 0.
    \]
    The latter is clearly satisfied for convolution kernels $K_b(t,r) = \widetilde{K}_b(t-r)$ and $K_{\sigma}(t,r) = \widetilde{K}_{\sigma}(t-r)$. Alternatively, for convolution kernels, a more direct proof may also be given via Young's inequality and similar arguments to \cite[Lemma 3.4]{PROMEL2023291}.} 
\end{Remark}

{
Consider the weighted occupation measure and self-intersection measures given by 
\begin{align*}
    \ell_t^{\delta}(A) &= \int_0^t \sigma_*(X_r)^{\delta} \1_{A}(X_r)\, \mathrm{d}r 
    \\ G_t^{\delta}(A) &= \int_0^t \int_0^t \sigma_*(X_{r_1})^{\delta} \sigma_*(X_{r_2})^{\delta} \1_A(X_{r_2} - X_{r_1})\, \mathrm{d}r_1 \mathrm{d}r_2,
\end{align*}
where $\delta \geq 0$. The following is our main result on the regularity of $\ell^{\delta}, G^{\delta}$, and the weighted law $\mu_t^{\delta}(A) = \E[w_*(X_t)^{\delta} \1_A(X_t)]$ for continuous solutions of \eqref{eq: CVSDE}. }

\begin{Theorem}\label{thm: CVSDE}
 Suppose that conditions (B1) and (B2) are satisfied, that $b \in C^{\beta_b}(\R^d)$, $\sigma \in C^{\beta_{\sigma}}(\R^d)$ for $\beta_b, \beta_{\sigma} \in [0,1]$, and that $g$ satisfies 
 \begin{align}\label{eq: g condition}
  \sup_{t \in [0,T]}\|g(t)\|_{L^p(\Omega)} < \infty \ \text{ and } \ 
  \| g(t) - g(s)\|_{L^p(\Omega)} \lesssim (t-s)^{\gamma_b \wedge \gamma_{\sigma}}
 \end{align}
 for some $p \in [2,\infty)$. Define $\zeta = \min\{ \beta_{b}(\gamma_b \wedge \gamma_{\sigma}) + \gamma_b, \ \beta_{\sigma}(\gamma_b \wedge \gamma_{\sigma}) + \gamma_{\sigma}\}$ and suppose that $H < \zeta$. Let $X$ be a continuous weak solution of \eqref{eq: CVSDE}. Then the following assertions hold:
 \begin{enumerate}
     \item[(a)] {Suppose that there exists $\eta \in (0,1/2)$ and $\delta \in [0,\eta/H]$ such that
 \[
    \sup_{t \in [0,T]}\E\left[ \sigma_*(X_t)^{p \left(\delta - \frac{\eta}{H}\right) } \right] < \infty.
 \]
 Define the regularity index 
 \begin{align*}
    \kappa_*(\eta) = \frac{1}{\zeta + \eta} \min\left\{ \theta (1 \wedge \delta)(\gamma_b \wedge \gamma_{\sigma}) \left(1 + \frac{\eta}{H}\right), \eta \left( \frac{\zeta}{H} - 1 \right) \right\}.
 \end{align*}
 Then for each $q \in [1,(1-\eta)p))$ and $\e \in (0, 1 - \eta - q/p)$, the weighted occupation measure of $X$ satisfies for $\kappa < \kappa_*(\eta) - d/q$
 \begin{align}\label{eq: regularity L}
   \ell^{\delta} \in L^{\frac{p}{q}}\left(\Omega, \P;\ C^{1 - \eta - \frac{q}{p} - \e}\left([0,T]; \mathcal{F}L_q^{\kappa}(\R^d)\right) \right).
 \end{align}
 If additionally $p > 2$, then the weighted self-intersection measure satisfies for $1 \leq q < (1-\eta)\frac{p}{2}$, $\varepsilon \in (0, 1 - \eta - 2q/p)$, and $\kappa < 2\kappa_*(\eta) - d/q$
 \begin{align}\label{eq: regularity G}
   G^{\delta} \in L^{\frac{p}{2q}}\left(\Omega, \P;\ C^{1 - \eta - \frac{2q}{p} - \e}\left([0,T]; \mathcal{F}L_{q}^{\kappa}(\R^d) \right) \right).
 \end{align} }

     \item[(b)] {If there exists $\delta \in [0,1]$ such that 
     \[
        \sup_{t \in [0,T]}\E\left[ \sigma_*(X_t)^{\delta-1}\right] < \infty,
     \]
     then there exists $0 \leq g_t \in L^1(\R^d)$ such that
     \begin{align}\label{eq: 7}
        \P[X_t \in \mathrm{d}x] = \1_{\{\sigma_*(x) > 0\}}g_t(x)dx + \widetilde{\mu}_t(\mathrm{d}x),
     \end{align}
     where $\widetilde{\mu}_t$ is supported on $\{ \sigma_* = 0 \}$. Moreover, there exists some $\kappa' > 0$ such that $g_t^{\delta}(x) := \sigma_*(x)^{\delta}g_t(x)$ satisfies $g_t^{\delta} \in B_{1,\infty}^{\kappa'}(\R^d)$ and $\| g_t^{\delta} \|_{B_{1,\infty}^{\kappa'}} \lesssim (1 \wedge t)^{- H}$ for $t \in (0,T]$. }
 \end{enumerate}
\end{Theorem}
\begin{proof}
 Let us verify conditions (A1) -- (A3). Assertion (a) is then a consequence of Theorem \ref{thm: VIto} and Corollary \ref{cor: selfintersection time}. Condition (A1) is a particular case of (B1). Define $b_t = b(X_t)$ and $\sigma_t = \sigma(X_t)$. Then an application of Proposition \ref{prop: CVSDE hoelder} gives  
 \[
    \|b(X_t) - \E[b(X_t) \ | \ \F_s] \|_{L^p(\Omega)} \leq 2 \|b(X_t) - b(X_s)\|_{L^p(\Omega)} \lesssim \|X_t - X_s\|^{\beta_b} \lesssim (t-s)^{\beta_b(\gamma_b \wedge \gamma_{\sigma})}
 \]
 and analogously $\|\sigma(X_t) - \sigma(X_s)\|_{L^p(\Omega)} \lesssim (t-s)^{\beta_{\sigma}(\gamma_b \wedge \gamma_{\sigma})}$. Thus condition (A2) holds for $\alpha_b = \beta_b (\gamma_b \wedge \gamma_{\sigma})$ and $\alpha_{\sigma} = \beta_{\sigma}(\gamma_b \wedge \gamma_{\sigma})$. By assumption (B2), it follows that condition (A3) holds for $\rho_s = \sigma_*(X_s)$. Then \eqref{eq: weight function bound} holds by assumption, and setting $\chi = (\gamma_b \wedge \gamma_{\sigma})\theta$, we find 
 \begin{align*}
     \|\rho_t - \rho_s\|_{L^{p}(\Omega)} 
     \lesssim \|\sigma_*(X_t) - \sigma_*(X_s)\|_{L^p(\Omega)}
     \lesssim \|X_t - X_s \|_{L^p(\Omega)}^{\theta}
     \lesssim (t-s)^{\chi},
 \end{align*}
 where we have used Proposition \ref{prop: CVSDE hoelder}. Hence, Theorem \ref{thm: VIto} and Corollary \ref{cor: selfintersection time} are applicable, which completes the proof of assertion (a). 
 
 {Concerning assertion (b), let $\mu_t^{\delta}(A) = \E[ \sigma_*(X_t)^{\delta} \1_A(X_t)]$. An application of Theorem \ref{thm: VIto density} shows that $\mu_t^{\delta}$ is absolutely continuous to the Lebesgue measure, and that its density $\widetilde{g}_t^{\delta}$ belongs to $B_{1,\infty}^{\kappa'}(\R^d)$ for some $\kappa' > 0$ and satisfies $\| \widetilde{g}_t^{\delta}\|_{B_{1,\infty}^{\kappa'}} \lesssim (1\wedge t)^{-H}$. For $n \geq 1$, and $A \in \mathcal{B}(\R^d)$ with Lebesgue measure zero, we find
 \[
    \P[ X_t \in A, \ \sigma_*(X_t) \geq 1/n ] \leq n^{1-\delta} \E[ \sigma_*(X_t)^{\delta - 1} \1_A(X_t) ] = n^{1-\delta} \mu_t^{\delta}(A) = 0.
 \]
 Letting $n \to \infty$ shows that $A \longmapsto \P[ X_t \in A, \ \sigma_*(X_t) > 0 ]$ is absolutely continuous with respect to the Lebesgue measure. Let $0 \leq g_t \in L^1(\R^d)$ be its density on $\{ \sigma_*(x) > 0 \}$ with $g_t \equiv 0$ on $\{ \sigma_*(x) = 0\}$. Then the representation \eqref{eq: 7} holds with $\widetilde{\mu}_t(A) = \E[ \1_{\{\sigma_*(X_t) = 0\}} \1_A(X_t)]$ and $\E[ \1_{\{\sigma_*(X_t) > 0\}} \1_A(X_t)] = \int_A g_t(x)\mathrm{d}x$. Finally, noting that $g_t^{\delta} = \widetilde{g}_t^{\delta}$, proves the desired regularity. }
\end{proof}

The choice $\beta_b = 0$ or $\beta_{\sigma} = 0$ corresponds to case where $b$ or $\sigma$, respectively, is only measurable with linear growth. Indeed, in such a case condition (A2) still holds true with $\alpha_b = 0$ or $\alpha_{\sigma} = 0$, respectively. {For the standard occupation measure and self-intersection measure, which corresponds to $\delta = 0$, we obtain, in view of Remark \ref{remark: delta zero}, the regularity index 
\[
    \kappa_*(\eta) = \frac{\eta}{\eta + \zeta}\left( \frac{\zeta}{H} - 1 \right).
\]
}

Finally, let us consider the two examples of stochastic Volterra processes below that illustrate these results. The first example addresses additive noise.

\begin{Example}
    Consider the equation
    \begin{align*}
    X_t = x_0 + \int_0^t \frac{(t-s)^{H-1/2}}{\Gamma(H + 1/2)}b(X_s)\, \mathrm{d}s + \int_0^t \frac{(t-s)^{H-1/2}}{\Gamma(H+1/2)}\, \mathrm{d}B_s
    \end{align*}
    where $H > 0$ and $b \in C^{\beta}(\R^d)$ for some $\beta \in [0,1]$. Then {we may take $\sigma_* \equiv 1$ and Theorem \ref{thm: CVSDE} is applicable with $\gamma_b = H+\frac{1}{2}$, $\gamma_{\sigma} = H$, $\zeta = H(1+\beta) + \frac{1}{2} > H$, $\delta = 0$, and yields the following regularity index for the occupation and self-intersection measures
    \[
        \kappa_*(\eta) = \frac{\eta}{\eta + H(1+\beta) + \frac{1}{2}} \left( \beta + \frac{1}{2H} \right). 
    \] 
    In particular, $\mu_t(A) = \P[X_t \in A]$ is absolutely continuous with respect to the Lebesgue measure, and its density $g_t$ satisfies $g_t \in B_{1,\infty}^{\kappa}(\R^d)$ for some $\kappa > 0$. Remark that $\kappa_*(\eta) \nearrow \infty$ as $H \searrow 0$, i.e. the local time is smooth whenever the fractional noise is sufficiently rough.}
\end{Example}

In our second example, we consider a variant of \eqref{eq: Volterra CIR} in which $\sqrt{x}$ is replaced by the power function $x^{\theta}$. The latter serves as a natural example for a diffusion coefficient that is degenerate at the boundary and merely H\"older continuous. 

\begin{Example}
    Consider in dimension $d = 1$ the equation
    \[
        X_t = x_0 + \int_0^t \frac{(t-s)^{H-1/2}}{\Gamma(H + \frac{1}{2})} (b + \beta X_s)\, \mathrm{d}s + \int_0^t \frac{(t-s)^{H-1/2}}{\Gamma(H + \frac{1}{2})} X_s^{\theta}\, \mathrm{d}B_s
    \]
    on $\R_+$, where $\theta \in [1/2, 1]$, $b \geq 0$ and $\beta \in \R$. This equation has a nonnegative, continuous weak solution due to \cite{MR4019885}. Theorem \ref{thm: CVSDE}.(a) is applicable {with $\gamma_b = H + \frac{1}{2}, \gamma_{\sigma} = H, \zeta = H(1+\theta)$, $\delta = \eta/H$, $\eta \in (0,1/2)$, and the weight function $\sigma_*(x) = 1 \wedge x^{\theta}$. This gives the following assertions:}
    \begin{enumerate}
        \item[(a)] The law of $X_t$ satisfies
        \[
            \P[ X_t \in \mathrm{d}x] = g_t(x)\mathrm{d}x + \P[X_t = 0]\delta_0(\mathrm{d}x),
        \]
        where $0 \leq g_t \in L^1(\R_+)$ is such that $g_t^{\theta}(x) = \1_{\R_+}(x)(1 \wedge x^{\theta})g_t(x)$ belongs to $B_{1,\infty}^{\kappa}(\R)$ for some $\kappa > 0$. 
    
        \item[(b)] The weighted occupation and self-intersection measure satisfy $\ell_t^{\eta/H}, G_t^{\eta/H} \in \mathcal{F}L_q^{\kappa}(\R) \subset H_{q/(q-1)}^{\kappa}(\R)$ a.s. for $\kappa < \kappa_*(\eta) - 1/q$, with $q \in [1,\infty)$ and the regularity index 
        \[
            \kappa_*(\eta) = \frac{1}{H(1+\theta) + \eta} \min\left\{ \theta (1\wedge \eta/H) H (1 + \eta/H),\ \eta \theta \right\} = \frac{\eta\theta}{H(1+\theta) + \eta}
        \]
        since $\eta \leq (1\wedge \eta/H)(H + \eta)$. In particular, since $q$ is arbitrary, we may take it sufficiently large such that $\kappa_*(\eta) > 1/q$. Hence $\ell_t^{\eta/H}, G_t^{\eta/H}$ admit a density in $H_{q/(q-1)}^{\kappa}(\R) \subset L^{1 + \frac{1}{q-1}}(\R)$ with $\kappa \in (0,\kappa_*(\eta) - 1/q)$.  
    \end{enumerate}
\end{Example}

In this example we have $\kappa_*(\eta) \nearrow \theta$ as $H \searrow 0$. Hence, the regularity index is bounded above due to the degeneracy at the boundary caused by the diffusion coefficient $\sigma(x) = \1_{\R_+}(x)x^{\theta}$. Finally, our results also apply to equations of the form 
\begin{align*}
    X_t &= g_X(t) + \int_0^t K_b^X(t,s)b^X(X_s, V_s)\, \mathrm{d}s
    \\ V_t &= g_V(t) + \int_0^t K_b^V(t,s)b^V(X_s, V_s)\, \mathrm{d}s + \int_0^t K_{\sigma}(t,s)\sigma(X_s, V_s)\, \mathrm{d}B_s
\end{align*}
where $K_b^X, K_b^V, K_{\sigma}$ are non-anticipating Volterra kernels. Such equations appear, e.g., in the study of Hamiltonian dynamics where $X$ denotes the position and $V$ the velocity/momentum of particles in the system. For such equations, we may apply our results to the Volterra It\^{o}-process $V$.

\subsection{Finite dimensional distributions}

The absolute continuity and regularity of the one-dimensional law of solutions of \eqref{eq: CVSDE} is studied in Theorem \ref{thm: CVSDE}. In this section, we focus on the absolute continuity of the finite-dimensional distributions for convolution-type equations of the form 
\begin{align}\label{eq: CVSDE convolution}
     X_t &= g(t) + \int_0^t k_b(t-s)b(X_s)\, \mathrm{d}s + \int_0^t k_{\sigma}(t-s)\sigma(X_s)\, \mathrm{d}B_s.
\end{align}
Here $k_b \in L^1_{\mathrm{loc}}(\R_+; \R^{d \times d})$ and $k_{\sigma} \in L_{\mathrm{loc}}^2(\R^d; \R^{d \times k})$, $b: \R^d \longrightarrow \R^d$, $\sigma: \R^d \longrightarrow \R^{k \times m}$, and $B$ an $m$-dimensional standard Brownian motion. While for Markov processes, absolute continuity of finite-dimensional distributions can be directly derived from the absolute continuity of their one-dimensional laws via the Chapman-Kolmogorov equations, such an approach is not directly applicable to stochastic Volterra processes. Below, we use Markovian lifts for \eqref{eq: CVSDE} to derive a lifted form of the Chapman-Kolmogorov equations sufficient for the absolute continuity of time-marginals. Let us suppose that the following set of conditions hold:
\begin{enumerate}
    \item[(C1)] $k_b \in L^1_{\mathrm{loc}}(\R_+; \R^{d \times d})$ and $k_{\sigma} \in L_{\mathrm{loc}}^2(\R^d; \R^{d \times k})$ are absolutely continuous on $(0,\infty)$, and there exists $\eta^* \in [0,2)$ with
    \[
        \int_0^{\infty} \left( |k_b'(t)|^2 + |k_{\sigma}'(t)|^2\right) (1\wedge t)^{\eta^*}\, \mathrm{d}t < \infty.
    \]
    \item[(C2)] Let $g: \R_+ \times \Omega \longrightarrow \R^d$ be $\mathcal{B}(\R_+) \otimes \mathcal{F}_0$-measurable, a.s. absolutely continuous on $(0,\infty)$, and there exist $\eta_g \in [0,1)$, $\eta_g \leq \eta^*$, and $p > 2$ with 
  \[
    \frac{1}{p} + \frac{\eta^*}{2} < \frac{1+\eta_g}{2}
  \]
  and
  \[
        \E\big[|g(1)|^{p}\big] + \E\left[ \left(\int_{0}^{\infty} |g'(t)|^2\hspace{0.02cm} (1 \wedge t)^{ \eta_g}\, \mathrm{d}t \right)^{p/2} \right] < \infty.
  \]
  \item[(C3)] $b: \R^d \longrightarrow \R^d$ and $\sigma: \R^d \longrightarrow \R^{k \times m}$ are  globally Lipschitz continuous.
\end{enumerate}

Let us first argue that \eqref{eq: CVSDE convolution} has a unique continuous strong solution. Indeed, by assumption (C1), an application of the mean-value theorem shows that 
\begin{align}\label{eq: k bound}
    &\ \int_0^T |k_b(t+h) - k_b(t)|\, \mathrm{d}t + \int_0^h |k_b(t)|\, \mathrm{d}t \lesssim h^{(\gamma+\frac{1}{2})\wedge 1},
    \\ &\int_0^T |k_{\sigma}(t+h) - k_{\sigma}(t)|^2 \, \mathrm{d}t + \int_0^h |k_{\sigma}(t)|^2\, \mathrm{d}t \lesssim h^{(2\gamma) \wedge 1} \notag
\end{align}
holds for $\gamma = 1 - \frac{\eta^*}{2}$. If $\eta^* \in [0,1)$ and additionally $k_b(0)=0$, $k_{\sigma}(0)=0$, then the exponents can be strengthened to $\gamma + 1/2$ and $2\gamma$. Likewise, by assumption (C2), we find 
\[
    \sup_{t \in [0,T]}\E[ |g(t)|^p ] < \infty \ \text{ and } \ \|g(t) - g(s)\|_{L^p(\Omega)} \lesssim (t-s)^{\frac{1 - \eta_g}{2}}.
\]
Hence, since $b,\sigma$ are globally Lipschitz continuous by (C3), the existence of a unique strong solution is clear. By Proposition \ref{prop: CVSDE hoelder} and Remark \ref{remark: convolution}, this solution admits a modification with continuous sample paths.

Following \cite[Section 6]{BBCF25}, let us construct a Markovian lift for \eqref{eq: CVSDE convolution}. For $\eta \geq 0$ let $\mathcal{H}_{\eta}$ be the space of all absolutely continuous functions $y: \R_+ \longrightarrow \R^d$ with finite norm
\[
 \|y\|_{\eta}^2 = |y(1)| + \int_0^{\infty} |y'(t)|^2 (1\wedge t)^{\eta}\, \mathrm{d}t < \infty.
\]
The shift semigroup $(S(t))_{t \geq 0}$ defined by $S(t)y(x) = y(x+t)$ is strongly continuous on $\mathcal{H}_{\eta}$, and $y(0) := \Xi y = y(1) - \int_0^1 y'(t)\, \mathrm{d}t$ is a bounded linear operator on $\mathcal{H}_{\eta}$ whenever $\eta \in [0,1)$. Define 
\[
    \mathcal{H} = \mathcal{H}_{\eta^*}, \ \mathcal{V} = \mathcal{H}_{\eta_g}, \ \rho = \frac{\eta^* - \eta_g}{2}.
\]
By \cite[Lemma 6.1]{BBCF25} combined with the assumption $\eta_g \leq \eta^* < 1 + \eta_g$, we find $\|S(t)\|_{L(\mathcal{H}, \mathcal{V})} \lesssim 1 + t^{-\rho}$ for $t > 0$ and $\Xi \in L(\mathcal{V}, \R^d)$. With this choice of $\mathcal{V}, \mathcal{H}$, let us define bounded linear operators $\xi_b, \xi_{\sigma}: \R^d \longrightarrow \mathcal{H}$ by $(\xi_b v)(t) = k_b(t)v$ and $(\xi_{\sigma} v)(t) = k_{\sigma}(t)v$, where $v \in \R^d$ and $t > 0$. Consider the corresponding abstract Markovian lift of \eqref{eq: CVSDE convolution} given by the following stochastic equation on $\mathcal{V}$:
\begin{align}\label{eq: abstract markovian lift}
    \mathcal{X}_t = S(t)g + \int_0^t S(t-s)\hspace{0.02cm}\xi_b\hspace{0.02cm} b(\Xi \mathcal{X}_s)\, \mathrm{d}s + \int_0^t S(t-s)\hspace{0.02cm}\xi_{\sigma} \hspace{0.02cm}\sigma(\Xi \mathcal{X}_s)\, \mathrm{d}B_s,
\end{align}
where $g \in L^p(\Omega, \mathcal{F}_0, \P; \mathcal{V})$ by assumption (C2). Finally, noting that $\frac{1}{p} + \rho < \frac{1}{2}$ and that $b,\sigma$ are Lipschitz continuous, it follows from \cite[Section 2]{BBCF25} that \eqref{eq: abstract markovian lift} has a unique solution $\mathcal{X}$ in $\mathcal{V}$ with continuous sample paths. In view of $g(\cdot) = \Xi S(\cdot)\xi$, $k_b(\cdot) = \Xi S(\cdot)\xi_b$, and $k_{\sigma}(\cdot) = \Xi S(\cdot)\xi_{\sigma}$, it is easy to see that $\Xi \mathcal{X} = X$. The following is our main result on the absolute continuity of finite-dimensional distributions of \eqref{eq: CVSDE convolution}

\begin{Theorem}\label{thm: fdd}
 Suppose that conditions (C1), (C2), (C3) are satisfied, and that there exists $H \in (0,1)$ and a constant $C_* > 0$ such that for each $h \in (0,1)$ and $\xi \in \R^d$ with $|\xi| = 1$
 \begin{align}\label{eq: LD}
        \int_0^h |k_{\sigma}(t)^{\top}\xi|^2\, \mathrm{d}t \geq C_* h^{2H}.
 \end{align}
 If $H + \eta^* < 2$, and there exists $\delta \in [0,1]$ such that for $\sigma_*(x) = 1 \wedge \lambda_{\min}(\sigma(x)\sigma(x)^{\top})$
 \[
    \sup_{t \in [0,T]}\E[ \sigma_*(X_t)^{\delta - 1}] < \infty,
 \]
 then for all $0 < t_1 < \dots < t_n \leq T$, the measure
 \[
    \mu_{t_1,\dots, t_n}^{\delta}(A) = \P\left[ (X_{t_1},\dots, X_{t_n}) \in A \cap \Gamma^{(n)} \right], \qquad A \in \mathcal{B}((\R^d)^n)
 \]
 is absolutely continuous with respect to the Lebesgue measure, where  
 \[
    \Gamma^{(n)} := \left\{ (x_1,\dots, x_n) \in (\R^d)^n \ : \ \prod_{j=1}^n \sigma_*(x_j) > 0  \right\}.
 \]
\end{Theorem}
\begin{proof}
 It suffices to show that
 \begin{align}\label{eq: assertion n}
    \E[ \1_{\Gamma^{(n)}}(X_{t_1},\dots, X_{t_n}) f(X_{t_1}, \dots, X_{t_n})] = 0
 \end{align}
 holds for each measurable function $f: (\R^d)^n \longrightarrow \R_+$ that satisfies $f \equiv 0$ a.e. with respect to the Lebesgue measure. We prove the assertion by induction over $n \in \N$. For $n = 1$, assumption (B1) holds by \eqref{eq: k bound} with $\gamma_b = (\gamma + \frac{1}{2})\wedge 1$ and $\gamma_{\sigma} = \gamma \wedge \frac{1}{2}$, while (B2) is satisfied by \eqref{eq: LD}. This gives $\zeta = 2 (\gamma_b \wedge \gamma_{\sigma}) = (2\gamma) \wedge 1$. Since $\zeta > H$ due to $H \in (0,1)$ and $\eta^* + H < 2$, Theorem \ref{thm: CVSDE}.(b) yields the absolute continuity of $\mu_t(A) = \P[X_t \in A \cap \Gamma^{(1)}]$. This proves \eqref{eq: assertion n} for $n = 1$. Now suppose that \eqref{eq: assertion n} holds for $n - 1$ with some fixed $n \geq 2$ and each measurable function $f: (\R^d)^{n-1} \longrightarrow \R_+$ such that $f \equiv 0$ holds a.e. with respect to the Lebesgue measure. To prove that the assertion also holds for $n$, we employ the Markovian lift $\mathcal{X}$ defined by \eqref{eq: abstract markovian lift}. 
 
 It follows from \cite[Theorem 2.4, Corollary 2.6]{BBCF25} that \eqref{eq: abstract markovian lift} admits a unique solution $\mathcal{X} \in L^p(\Omega, \P; C(\R_+; \mathcal{V}))$ that is a $C_b(\mathcal{V})$-Feller process and satisfies $X = \Xi \mathcal{X}$. Let $p_t(z, \mathrm{d}y)$ be its transition probabilities on $\mathcal{V}$. To shorten the notation, let $\widetilde{f}(x_1,\dots, x_n) = \1_{\Gamma^{(n)}}(x_1, \dots, x_n) f(x_1, \dots, x_n)$. The Markov property for path-dependent functionals applied to $(z_1,\dots, z_n) \longmapsto \widetilde{f}(\Xi z_1, \dots, \Xi z_n)$ yields
 \begin{align*}
     \E[\widetilde{f}(X_{t_1},\dots, X_{t_n})] 
     &= \E\left[ \E\left[ \widetilde{f}(\Xi \mathcal{X}_{t_1}, \dots, \Xi \mathcal{X}_{t_n}) \ | \ \mathcal{F}_{t_{n-1}}\right] \right]
     \\ &= \E\left[ \int_{\mathcal{V}} \widetilde{f}(\Xi\mathcal{X}_{t_1},\dots, \Xi \mathcal{X}_{t_{n-1}}, \Xi z) p_{t_n - t_{n-1}}(\mathcal{X}_{t_{n-1}}, \mathrm{d}z) \right]
     \\ &= \E\left[ G(X_{t_1}, \dots, X_{t_{n-1}}, \mathcal{X}_{t_{n-1}}) \right],
 \end{align*}
 where the function $G$ is defined by
 \[
    G(x_1,\dots, x_{n-1}, z_{n-1}) = \int_{\mathcal{V}} \widetilde{f}(x_1,\dots, x_{n-1}, \Xi z) p_{t_n - t_{n-1}}(z_{n-1}, \mathrm{d}z).
 \]
 By the disintegration of measures, we may write, for any measurable function $F: \mathcal{V} \longrightarrow \R_+$, 
 \begin{align*}
  \int_{\mathcal{V}}F(y)\, p_t(z,\mathrm{d}y) 
  = \int_{\R^d}\int_{\mathcal{V}} F(y)p_t^{0}(z, x , \mathrm{d}y)p_t^1(z,\mathrm{d}x),
 \end{align*}
  where $p_t^1(z,\mathrm{d}x) = \P_z[ \Xi \mathcal{X}_t \in \mathrm{d}x]$ denotes the marginal of $\mathcal{X}_t$ with respect to the projection $\Xi$, and $p_t^0(z,x, \mathrm{d}y) = \P_z[ \mathcal{X}_t \in \mathrm{d}y \ | \ \Xi \mathcal{X}_t = x]$ denotes the law of $\mathcal{X}_t$ conditional on $\Xi \mathcal{X}_t = x$. Hence we obtain
 \begin{align} \notag
     G(x_1,\dots, x_{n-1}, z_{n-1}) &= \int_{\R^d}\int_{\mathcal{V}} \widetilde{f}(x_1,\dots, x_{n-1}, \Xi z_n) p_{t_n - t_{n-1}}^{0}(z_{n-1}, x_n, \mathrm{d}z_n)p_{t_n - t_{n-1}}^1(z_{n-1},\mathrm{d}x_n) 
     \\ &= \int_{\R^d} \widetilde{f}(x_1,\dots, x_{n-1}, x_n) p_{t_n - t_{n-1}}^1(z_{n-1},\mathrm{d}x_n) \label{eq: G rep}
 \end{align}
 since $p_{t_n - t_{n-1}}^0(z_{n-1}, x_n, \mathrm{d}z_n)$ is supported on $\{y \in \mathcal{V} \ : \ \Xi y = x_n\}$ and $p_{t_n - t_{n-1}}^0(z_{n-1}, x_n, \mathcal{V}) = 1$ for $p_t^1(z,\cdot)$-a.a. $x_n$ by definition of regular conditional distributions. Let us denote by 
 \[
    \P\left[ \mathcal{X}_{t_{n-1}} \in \mathrm{d}z_{n-1} \ | X_{t_1} = x_1, \dots, X_{t_{n-1}} = x_{n-1} \right]
 \]
 the regular conditional distribution of $\mathcal{X}_{t_{n-1}}$ given $X_{t_1} = x_1, \dots, X_{t_{n-1}} = x_{n-1}$. Define
 \begin{align*}
    &\ h(x_1,\dots, x_{n-1}) 
    \\ &\qquad = \int_{\mathcal{V}} \left( \int_{\R^d}\widetilde{f}(x_1,\dots, x_n) p_{t_n - t_{n-1}}^1(z_{n-1}, \mathrm{d}x_n) \right) \P\left[ \mathcal{X}_{t_{n-1}} \in \mathrm{d}z_{n-1} \ | X_{t_1} = x_1, \dots, X_{t_{n-1}} = x_{n-1} \right].
 \end{align*}
 Then $h: (\R^d)^{n-1} \longrightarrow \R_+$ is measurable and uniquely determined up to sets of measure zero with respect to the law of $(X_{t_1}, \dots, X_{t_{n-1}})$. By the law of total expectation, the disintegration property of regular conditional probabilities, and \eqref{eq: G rep}, we find
  \begin{align*}
     \E\left[ G(X_{t_1}, \dots, X_{t_{n-1}}, \mathcal{X}_{t_{n-1}}) \right]
     &= \E \left[ \E \left[ G(X_{t_1}, \dots, X_{t_n-1}, \mathcal{X}_{t_{n-1}}) \ | X_{t_1}, \dots, X_{t_{n-1}} \right] \right ] 
     \\ &= \E[ h(X_{t_1}, \dots, X_{t_{n-1}}) ].
 \end{align*}
 Since $f \equiv 0$ a.e. by assumption, Fubini's theorem yields the existence of a Lebesgue null set $N \subset (\R^d)^{n-1}$ such that for all $(x_1, \dots, x_{n-1}) \in (\R^d)^{(n-1)} \setminus N$, the section $x_n \longmapsto f(x_1, \dots, x_n)$ vanishes a.e. with respect to the Lebesgue measure on $\R^d$. Furthermore, by the base case $n=1$, the measure $A \longmapsto p_{t_n - t_{n-1}}^1(z_{n-1}, A \cap \Gamma^{(1)})$ is absolutely continuous with respect to the Lebesgue measure with density $g_{t_n - t_{n-1}}(z_{n-1}, \cdot)$. Finally, observe that $\1_{\Gamma^{(n)}}(x_1,\dots, x_n) = \1_{\Gamma^{(n-1)}}(x_1,\dots, x_{n-1})\1_{\Gamma^{(1)}}(x_n)$ due to the multiplicative structure in the definition of $\Gamma^{(n)}$. Consequently, for all $(x_1, \dots, x_{n-1}) \in (\R^d)^{(n-1)} \setminus N$ we find
 \begin{align*}
    \int_{\R^d}\widetilde{f}(x_1,\dots, x_n) p_{t_n - t_{n-1}}^1(z_{n-1}, \mathrm{d}x_n) 
    &= \1_{\Gamma^{(n-1)}}(x_1,\dots, x_{n-1}) \int_{\Gamma^{(1)}} f(x_1,\dots, x_n) g_{t_n - t_{n-1}}(z_{n-1}, x_n)\, \mathrm{d}x_n
    \\ &= 0
 \end{align*}
 for all $z_{n-1} \in \mathcal{V}$. By definition of $h$, this implies that $h(x_1, \dots, x_{n-1}) = 0$ for almost all $(x_1, \dots, x_{n-1}) \in (\R^d)^{(n-1)}$, and hence $\E[ h(X_{t_1}, \dots, X_{t_{n-1}}) ] = 0$ by induction hypothesis. The assertion now follows from
 \begin{align*}
     \E[ \widetilde{f}(X_{t_1}, \dots, X_{t_n})] &= \E[ G(X_{t_1}, \dots, X_{t_{n-1}}, \mathcal{X}_{t_n})] = \E[ h(X_{t_1}, \dots, X_{t_{n-1}}) ] = 0.
 \end{align*}
\end{proof}

The case of fractional kernels provides a natural example for this theorem.

\begin{Example}
    Let $k_b(t) = \frac{(t-s)^{H_b - \frac{1}{2}}}{\Gamma(H_b + \frac{1}{2})}\mathrm{id}_{\R^d}$ and $k_{\sigma}(t) = \frac{(t-s)^{H_{\sigma} - \frac{1}{2}}}{\Gamma(H_{\sigma} + \frac{1}{2})}\mathrm{id}_{\R^d}$, where $H_b, H_{\sigma} \in (0,1)$. Then condition (C1) is satisfied for each $\eta^* > 2 - 2 H_b \wedge H_{\sigma}$, while condition \eqref{eq: LD} is satisfied for $H = H_{\sigma}$. In particular, $p, \eta_g$ appearing in condition (C2) shall satisfy
    \[
        1 + \frac{1}{p} < \frac{1 + \eta_g}{2} + H_b \wedge H_{\sigma}.
    \]
\end{Example}

Let us close with a few remarks on the applicability of this theorem and possible limitations. First, the Lipschitz continuity of $ b$ and $ \sigma$ is not essential. The key arguments can be applied whenever we can construct some (not necessarily unique) Markovian lift from \eqref{eq: abstract markovian lift} that has Markov transition probabilities supported on $\mathcal{V}$. Since existence for \eqref{eq: CVSDE convolution} yields the existence for \eqref{eq: abstract markovian lift}, here one may either seek for uniqueness in law as established in \cite{MR4897765}, or apply the Markov selection theorem. Likewise, the Markovian lift introduced and studied in \cite{BBCF25} applies to non-convolution kernels. Hence, the additional convolution structure assumed in \eqref{eq: CVSDE convolution} is not essential and can be removed.

\section{Self-intersecting diffusion equations}

\subsection{Nonlinear Young integration}

Let $(E, \|\cdot\|_E)$ be a Banach space. Below, we provide a summary of the constructions given in \cite{BHR23, H20} for the two-parameter case $[0,T]^2$. For $s,t \in [0,T]^2$ we let $s < t$ be the partial ordering defined by $s_1 < t_1$ and $s_2 < t_2$, and set $[s,t] = [s_1, t_1] \times [s_2, t_2]$. For a function $f: [0,T]^2 \longrightarrow E$, define
\[
    \Box_{s,t}f = f_{t_1,t_2} - f_{t_1,s_2} - f_{s_1,t_2} + f_{s_1,s_2},
\]
and when $f: [0,T]^2 \times [0,T]^2 \longrightarrow E$ let $\Box_{s,t}f = f_{(s_1,s_2), (t_1,t_2)}$. For $\alpha \in (0,1)^2$ we let $C^{\alpha}([0,T]^2; E)$ be the space of functions $f: [0,T]^2 \longrightarrow E$ such that $[f]_{C^{\alpha}([0,T]^2; E)} = [f]_{(1,0), \alpha} + [f]_{(0,1), \alpha} + [f]_{(1,1),\alpha} < \infty$ where
\begin{align*}
    [f]_{(1,0),\alpha} &= \sup_{s \neq t} \frac{\|f_{t_1,s_2} - f_{s_1,s_2}\|_E}{|t_1 - s_1|^{\alpha_1}},
    \\ [f]_{(0,1),\alpha} &= \sup_{s \neq t} \frac{\|f_{s_1,t_2} - f_{s_1,s_2}\|_E}{|t_2 - s_2|^{\alpha_2}},
    \\ [f]_{(1,1),\alpha} &= \sup_{s \neq t} \frac{\|\Box_{s,t}f\|_E}{|t_1-s_1|^{\alpha_1}|t_2 - s_2|^{\alpha_2}}.
\end{align*}
Equipped with the norm $\|f\|_{C^{\alpha}([0,T]^2; E)} = \|f\|_{\infty} + [f]_{C^{\alpha}([0,T]^2; E)}$, $C^{\alpha}([0,T]^2;E)$ becomes a Banach space. If $\alpha_1 = \alpha_2 = \alpha$, then we also write $C^{\alpha}([0,T]^2; E) = C^{(\alpha,\alpha)}([0,T]^2; E)$. 

The boundary operators are for $f: [0,T]^2 \times [0,T]^2 \longrightarrow E$ and $s < r < t$ in $[0,T]^2$ defined by
\begin{align*}
    \delta_{r_1}^1 f_{s,t} &= f_{s,t} - f_{s,(r_1,t_2)} - f_{(r_1,s_2),t},
    \\ \delta_{r_2}^2 f_{s,t} &= f_{s,t} - f_{s,(t_1,r_2)} - f_{(s_1,r_2), t}.
\end{align*}
The composition of these operators is then given by
\[
 \delta_r f = \delta_{r_1}^1 \delta_{r_2}^2 f = \delta_{r_2}^2 \delta_{r_1}^1 f.
\]
For given $\alpha \in (0,1)^2$ and $\beta \in (1,\infty)^2$ we let $C_2^{\alpha,\beta}([0,T]^2; E)$ denote the space of all functions $A: [0,T]^2 \times [0,T]^2 \longrightarrow E$ such that $A_{s,t} = 0$ when $s_1 = t_1$ or $s_2 = t_2$, and $[A]_{\alpha} + [\delta A]_{\alpha,\beta} < \infty$ where
\[
    [\delta A]_{\alpha,\beta} = [A]_{(1,0),\alpha, \beta} + [A]_{(0,1),\alpha,\beta} + [A]_{(1,1),\alpha,\beta}
\]
and the remaining terms are defined by
\begin{align*}
    [A]_{\alpha} &= \sup_{s \neq t} \frac{\|A_{s,t}\|_E}{|t_1-s_1|^{\alpha_1}|t_2 - s_2|^{\alpha_2}},
    \\ [A]_{(1,0),\alpha, \beta} &= \sup_{s < r < t} \frac{\|\delta_{r_1}^1 A_{s,t}\|_E}{|t_1 - s_1|^{\beta_1}|t_2-s_2|^{\alpha_2}},
    \\ [A]_{(0,1),\alpha, \beta} &= \sup_{s < r < t} \frac{\|\delta_{r_2}^2 A_{s,t}\|_E}{|t_1 - s_1|^{\alpha_1}|t_2-s_2|^{\beta_2}},
    \\ [A]_{(1,1),\alpha, \beta} &= \sup_{s < r < t} \frac{\|\delta_r A_{s,t}\|_E}{|t_1 - s_1|^{\beta_1}|t_2-s_2|^{\beta_2}}.
\end{align*}

The two-parameter sewing lemma states that for each $A \in C_2^{\alpha,\beta}([0,T]^2 \times [0,T]^2; E)$ satisfying $\alpha \in (0,1)^2$ and $\beta \in (1,\infty)^2$ there exists a unique $\mathcal{I}(A) \in C^{\alpha}([0,T]^2; E)$ obtained as the limit of Riemann-Stieltjes sums
\[
    \mathcal{I}(A)_{s,t} = \lim_{|\mathcal{P}_{s,t}| \to 0}\sum_{[u,v] \in \mathcal{P}_{s,t}}A_{u,v}, \qquad (s,t) \in [0,T]^2
\]
where $\mathcal{P}_{s,t}$ is a partition of $[s,t]$ by cubes $[u,v] \subset [0,T]^2$. Moreover, setting $\mathcal{I}(A)_t := \mathcal{I}(A)_{0,t}$, we find $\left( t \longmapsto \mathcal{I}(A)_t\right) \in C^{\alpha}([0,T]^2;E)$ and $\Box_{s,t}\mathcal{I}(A) = \mathcal{I}(A)_{s,t}$. Finally, there exists a constant $C = C(\alpha,\beta, T) > 0$ such that for $s,t \in [0,T]^2$ with $s < t$
\[
    \| \mathcal{I}(A)_{s,t} - A_{s,t}\|_E \leq C |t_1 - s_1|^{\alpha_1}|t_2 - s_2|^{\alpha_2}\left( |t_1 - s_1|^{\beta_1 - \alpha_1} + |t_2 - s_2|^{\beta_2 - \alpha_2}\right)[\delta A]_{\alpha,\beta}.
\]

Now let $E = \R^d$. The two-parameter sewing lemma allows us to construct two-parameter nonlinear Young integrals from their local approximations $\Box_{u,v} A(x_u)$ where $A \in C^{\gamma}([0,T]^2; C^{1+\kappa}(\R^d))$ with $\gamma \in (1/2,1]^2$, $\kappa \in (0,1]$, and $x \in C^{\beta}([0,T]^2; \R^d)$ with $\beta \in (0,1)^2$ satisfies $\gamma_i + \beta_i\kappa > 1$, see \cite{BHR23}. Below, we state a particular case essential for the study of \eqref{eq: selfintersecting Volterra diffusion process}. 

\begin{Lemma}\label{lemma two parameter nonlinear Young}
    Let $\gamma \in (1/2,1]$, $\beta \in (0,1)$, $\kappa \in (0,1]$, $A \in C^{\gamma}([0,T]^2; \mathcal{C}^{1+\kappa}(\R^d))$, $\theta \in C^{\beta}([0,T]; \R^d)$, and suppose that $\gamma + \kappa \beta > 1$. Then the following Riemann sums are convergent
    \[
        \int_0^{t_2} \int_0^{t_1} A(\mathrm{d}r, \theta_{r_2} - \theta_{r_1}) = \lim_{|\mathcal{P}_{t}| \to 0}\sum_{[u,v] \in \mathcal{P}_{s,t}} \Box_{u,v}A(\cdot, \theta_{u_2} - \theta_{u_1})
    \]
    where $\mathcal{P}_t$ denotes a partition of $[0,t]$ via cubes $[u,v]$. We have $(t_1,t_2) \longmapsto \int_0^{t_1} \int_0^{t_2} A(\mathrm{d}r, \theta_{r_2} - \theta_{r_1}) \in C^{\gamma}([0,T]^2; \R^d)$, and there exists a constant $C > 0$ such that
    \begin{align*}
         \bigg| \int_{s_2}^{t_2}\int_{s_1}^{t_1}& A(\mathrm{d}r, \theta_{r_2} - \theta_{r_1}) - \Box_{s,t} A(\cdot, \theta_{s_2} - \theta_{s_1}) \bigg| 
         \\ &\leq C \| A\|_{C_T^{\gamma}C_x^{1+\kappa}} [\theta]_{C^{\beta}([0,T])}^{1 + \kappa} (t_2-s_2)^{\gamma}(t_1-s_1)^{\gamma}\left[ |t_2 - s_2|^{\kappa \beta} + |t_1 - s_1|^{\kappa \beta} \right].
    \end{align*}
    Suppose that $A, \widetilde{A} \in C^{\gamma}([0,T]^2; \mathcal{C}^{2+\kappa}(\R^d))$ and $\theta, \widetilde{\theta} \in C^{\beta}([0,T]; \R^d)$, where $\gamma \in (1/2,1]$ and $\gamma + \kappa \beta > 1$. Then there exists a constant $C > 0$ such that the corresponding nonlinear Young integrals satisfy
    \begin{align*}
         \bigg| \int_{s_2}^{t_2}\int_{s_1}^{t_1}& A(\mathrm{d}r, \theta_{r_2} - \theta_{r_1}) - \int_{s_2}^{t_2}\int_{s_1}^{t_1}\widetilde{A}(\mathrm{d}r, \widetilde{\theta}_{r_2} - \widetilde{\theta}_{r_1}) \bigg| 
        \\ &\lesssim \max\{ [ \theta]_{C^{\beta}([0,T])}, [ \widetilde{\theta}]_{C^{\beta}([0,T])} \}  \|A - \widetilde{A}\|_{C_T^{\gamma} C_x^{2+\kappa}} (t_1-s_1)^{\gamma}(t_2-s_2)^{\gamma}
        \\ &\qquad + K\left( \|\theta - \widetilde{\theta}\|_{C^{\beta}([s_1,t_1])} + \|\theta - \widetilde{\theta}\|_{C^{\beta}([s_2,t_2])} \right)(t_1-s_1)^{\gamma}(t_2-s_2)^{\gamma}
    \end{align*}
    where the constant $K$ is given by
    \[
    K = \left( [ \theta]_{C^{\beta}([0,T])} + [ \widetilde{\theta}]_{C^{\beta}([0,T])}\right)^{1+\kappa} \max\{\|A\|_{C^{\gamma}_T C_x^{2+\kappa}},\| \widetilde{A}\|_{C_T^{\gamma}C_x^{2+\kappa}} \}.
    \]
\end{Lemma}
\begin{proof}
    Define the two parameter function $\Theta_{u} := \theta_{u_2} - \theta_{u_1}$, $u \in [0,T]^2$. Then $\Box_{u,v} \Theta = 0$ and hence we obtain $\Theta \in C^{\beta}([0,T]^2; \R^d)$ with $[\Theta]_{C^{\beta}([0,T]^2;\R^d)} \leq 2[ \theta]_{C^{\beta}([0,T]; \R^d)}$. The first assertion follows from \cite[Proposition 16]{BHR23} while the stability bound follows from \cite[Proposition 19]{BHR23}.
\end{proof}

\subsection{Stochastic equations with distributional self-intersections}

In this section, we study the solution theory for \eqref{eq: selfintersecting Volterra diffusion process}. Here and below, we fix a distributional drift $b \in \mathcal{S}'(\R^d)$ and treat the noise $Z$ path-by-path, i.e. we fix some realisation of $z \in C([0,T]; \R^d)$ of $Z$, and study \eqref{eq: selfintersecting Volterra diffusion process} with $Z$ replaced by one fixed choice of $z$. Denote by $G$ its two-parameter self-intersection measure defined by 
\begin{align}\label{eq: selfintersection measure two parameter}
    G_{t_1,t_2}(C) = \int_{0}^{t_2}\int_{0}^{t_1} w_{r_1}w_{r_2} \1_{C}(z_{r_2} - z_{r_1})\, \mathrm{d}r_1 \mathrm{d}r_2,
\end{align}
where $t = (t_1,t_2) \in [0,T]^2$ and $w \in L^{\infty}([0,T])$ denotes a weight function. Let us define the reflected measure by $\overline{G}_{t_1,t_2}(C) = G_{t_1,t_2}(-C)$, and set 
\begin{align}\label{eq: A}
    A^{b}(t_1,t_2, \theta) = (b \ast \overline{G}_{t_1,t_2})(\theta), \qquad \theta \in \R^d.
\end{align}
Then $A^{b}$ is a two-parameter family of functions which satisfies 
\begin{align*}
    A^{b}(t_1,t_2,\theta) = \int_0^{t_2}\int_0^{t_1}w_{r_1}w_{r_2} b(\theta + (x_{r_2} - x_{r_1}))\, \mathrm{d}r_1 \mathrm{d}r_2
\end{align*}
whenever $b$ is bounded and measurable. Denote by $\int_0^t \int_0^t A^b(\mathrm{d}r, \theta_{r_2} - \theta_{r_1})$ the nonlinear Young integral evaluated at $t = t_1 = t_2$ given by Lemma \ref{lemma two parameter nonlinear Young}. 

\begin{Definition}
    Let $b \in \mathcal{S}'(\R^d)$ and $z \in C([0,T])$ such that $A^b \in C^{\gamma}([0,T]^2; \mathcal{C}^{1+\kappa}(\R^d))$ for some $\gamma \in (1/2,1]$ and $\kappa \in (0,1]$. A solution of \eqref{eq: selfintersecting Volterra diffusion process} is a function $u = \theta + z \in C^{\gamma}([0,T]; \R^d)$ with $\gamma(1+\kappa) > 1$ such that $\theta$ solves the nonlinear Young equation
        \begin{align}\label{eq; nonlinear Young 2}
        \theta_t = u_0 + \int_0^{t} \int_0^{t} A^{b}(\mathrm{d}r, \theta_{r_2} - \theta_{r_1}).
    \end{align}
\end{Definition}

We show that for regular $b$, this definition coincides with the classical definition of solutions for \eqref{eq: selfintersecting Volterra diffusion process}. Moreover, when $A^b$ is sufficiently regular, we prove the existence and uniqueness of solutions of \eqref{eq; nonlinear Young 2}, and hence of \eqref{eq: selfintersecting Volterra diffusion process}.

\begin{Proposition}\label{thm: rough equation 2}
    Let $z \in C([0,T]; \R^d)$ and $b \in \mathcal{S}'(\R^d)$. Then the following assertions hold:
    \begin{enumerate}
        \item[(a)] If $b \in C_b(\R^d)$ and $A^{b} \in C^{\gamma}([0,T]^2; \mathcal{C}^{1+\kappa}(\R^d))$ for $\gamma \in (1/2,1]$ and $\kappa \in (0,1)$ satisfying $\gamma + \kappa > 1$, then $u \in C([0,T]; \R^d)$ is a solution of \eqref{eq: selfintersecting Volterra diffusion process} if and only if $\theta_t = u(t) - z_t$ is Lipschitz continuous and solves the nonlinear Young equation \eqref{eq; nonlinear Young 2}.

        \item[(b)] If $A^{b}\in C^{\gamma}([0,T]^2; \mathcal{C}^{2+\kappa}(\R^d))$ for $\gamma \in (1/2,1]$ and $\kappa \in (0,1)$ satisfying $\gamma\left(1 + \kappa\right) > 1$, then \eqref{eq: selfintersecting Volterra diffusion process} has a unique solution $u \in C([0,T]; \R^d)$, i.e. $\theta = u - z \in C^{\gamma}([0,T]; \R^d)$ solves the nonlinear Young equation \eqref{eq; nonlinear Young 2}.
    \end{enumerate}
\end{Proposition}
\begin{proof}
    (a) Suppose that $u$ is a solution of \eqref{eq: selfintersecting Volterra diffusion process}, then $\theta = u - z$ satisfies \eqref{eq: 2}. In particular, since $b$ is bounded, it follows that $\theta$ is Lipschitz continuous. By assumption on $A^{b},$ the two-parameter nonlinear Young integral is well-defined. Define the two-parameter function
    \[
        \mathcal{A}_{t_1, t_2} = \int_0^{t_2} \int_0^{t_1} w_{r_2} w_{r_1} b( (\theta_{r_2} - \theta_{r_1}) + (z_{r_2} - z_{r_1}))\, \mathrm{d}r_2 \mathrm{d}r_1.
    \]
    Let us show that $\int_0^{t_2}\int_0^{t_1}A^{b}(\mathrm{d}r, \theta_{r_2} - \theta_{r_1}) = \mathcal{A}_{t_1,t_2}$. Let $\e > 0$. Since $\theta$ and $z$ are continuous $[0,T]$, their trajectories are compact. Hence, by continuity of $b$ we may take $|v_2 - u_2| + |v_1 - u_1|$ sufficiently small enough such that
    \begin{align*}
        |b( (\theta_{r_2} - \theta_{r_1}) + (z_{r_2} - z_{r_1})) - b( (\theta_{u_2} - \theta_{u_1}) + (z_{r_2} - z_{r_1}))| < \e
    \end{align*}
    holds for $r_1 \in [u_1, v_1]$ and $r_2 \in [u_2, v_2]$. Hence using the definition of $A^{b}(\cdot, \theta_{u_2} - \theta_{u_1})$ we obtain
    \begin{align*}
        \left|\Box_{u,v} A^{b}(\cdot, \theta_{u_2} - \theta_{u_1}) - \Box_{u,v}\mathcal{A} \right|
        \lesssim \e (v_2 - u_2)(v_1 - u_1).
    \end{align*}
    From this we obtain for a partition $\mathcal{P}_{t}$ of $[0,t] = [0,t_1] \times [0,t_2]$ 
    \begin{align*}
        \sum_{[u,v] \in \mathcal{P}_t } \left|\Box_{u,v} A^{b}(\cdot, \theta_{u_2} - \theta_{u_1}) - \Box_{u,v}\mathcal{A} \right|
        \lesssim \e \sum_{[u,v] \in \mathcal{P}_t}(v_2 -u_2)(v_1 - u_1) = \e t_1 t_2.
    \end{align*}
    Letting first the mesh size of the partition go to zero, and then $\e \searrow 0$, shows that 
    \begin{align*}
        \int_0^{t_2} \int_0^{t_1} A^{b}(\mathrm{d}r, \theta_{r_2} - \theta_{r_1}) &= \lim_{|\mathcal{P}_{t}| \to 0} \sum_{[u,v] \in \mathcal{P}_{t}} \Box_{u,v} A^{b}(\cdot, \theta_{u_2} - \theta_{u_1})
        \\ &= \lim_{|\mathcal{P}_{t}| \to 0} \sum_{[u,v] \in \mathcal{P}_{t}} \Box_{u,v}\mathcal{A}
        \\ &= \mathcal{A}_{t_1, t_2}
    \end{align*}
    where the first equality follows from the construction of the nonlinear Young integral, and the last relation follows by telescoping summation since $\mathcal{P}_{t}$ is a partition of $[0,t_1] \times [0,t_2]$ by rectangles. Hence $\theta$ is a solution of the corresponding nonlinear Young equation. The converse statement follows in the same way.

    (b) Since equation \eqref{eq; nonlinear Young 2} is evaluated along a single time parameter $t$, while the integration is over the two-parameter domain $[0,t]^2$ against the difference $\theta_{r_2} - \theta_{r_1}$, we cannot directly apply the 2D field solutions of \cite[Theorem 21]{BHR23}. Below, we show how the usual fixed-point argument can be applied to this setting. For a small time horizon $\tau \in (0,T]$, let $\mathcal{X}_\tau = C^\gamma([0,\tau]; \R^d)$ equipped with the standard H\"older norm $\| \cdot \|_{C^{\gamma}([0,\tau])}$, and let $\mathcal{B}_R$ be the closed ball of radius $R > 0$ centered at the constant path $\theta \equiv u_0$. For $\theta \in \mathcal{B}_R$, define the map $\mathcal{T}$ by
    \begin{align*}
        \mathcal{T}_t(\theta) = u_0 + \int_0^t \int_0^t A^b(\mathrm{d}r, \theta_{r_2} - \theta_{r_1}), \quad t \in [0,\tau].
    \end{align*}
    We show that there exist $R$ and $\tau$ such that $\Gamma$ maps $\mathcal{B}_R$ into itself and is a contraction on $\mathcal{B}_R$.
    
    Firstly, by the two-parameter nonlinear Young integration bounds given in Lemma \ref{lemma two parameter nonlinear Young}, for any rectangle $[u_1,v_1] \times [u_2,v_2] \subset [0,\tau]^2$, we have
    \begin{align*}
        \left| \int_{u_1}^{v_1}\int_{u_2}^{v_2} A^b(\mathrm{d}r, \theta_{r_2} - \theta_{r_1}) \right| 
        &\leq \left| \int_{u_1}^{v_1}\int_{u_2}^{v_2} A^b(\mathrm{d}r, \theta_{r_2} - \theta_{r_1}) - \Box_{u,v} A^b(\theta_{u_2} - \theta_{u_1}) \right| + | \Box_{u,v}A^b(\theta_{u_2} - \theta_{u_1})| 
        \\ &\lesssim \|A^b\|_{C^\gamma_{\tau} \mathcal{C}_x^{1+\kappa}} [\theta]_{C^{\gamma}([0,\tau])}^{1+\kappa} |v_2 - u_2|^{\gamma}|v_1 - u_1|^{\gamma} \left( |v_2 - u_2|^{\kappa \gamma} + |v_1 - u_1|^{\kappa \gamma} \right) 
        \\ &\qquad + \|A^b\|_{C_{\tau}^\gamma \mathcal{C}_x^{1+\kappa}} |v_2 - u_2|^{\gamma}|v_1 - u_1|^{\gamma}
        \\ &\lesssim \|A^b\|_{C^\gamma_{\tau} \mathcal{C}_x^{1+\kappa}}\left( 1 + 2\tau^{\gamma \kappa}[ \theta]_{C^{\gamma}([0,\tau])}^{1+\kappa} \right) |v_2 - u_2|^{\gamma}|v_1 - u_1|^{\gamma}
        \\ &\lesssim \|A^b\|_{C^\gamma_{\tau} \mathcal{C}_x^{1+\kappa}}\left( 1 + 2T^{\gamma \kappa}R^{1+\kappa}\right) |v_2 - u_2|^{\gamma}|v_1 - u_1|^{\gamma},
    \end{align*}
    where we have used $v_1 - u_1, v_2 - u_2 \leq \tau \leq T$ and that $\theta \in \mathcal{B}_R$ which gives $[ \theta]_{C^{\gamma}([0,\tau])} = [ \theta - u_0]_{C^{\gamma}([0,\tau])} \leq \| \theta - u_0\|_{C^{\gamma}([0,\tau])} \leq R$. Applying this bound to the increment $\mathcal{T}_t(\theta) - \mathcal{T}_s(\theta)$ with $0 \leq s < t \leq \tau$ gives
    \begin{align*}
        |\mathcal{T}_t(\theta) - \mathcal{T}_s(\theta)| &\leq \left| \int_{s}^{t}\int_{0}^{s} A^b(\mathrm{d}r, \theta_{r_2} - \theta_{r_1}) \right| + \left| \int_{0}^{s}\int_{s}^{t} A^b(\mathrm{d}r, \theta_{r_2} - \theta_{r_1}) \right| + \left| \int_{s}^{t}\int_{s}^{t} A^b(\mathrm{d}r, \theta_{r_2} - \theta_{r_1}) \right| 
        \\ &\lesssim \|A^b\|_{C^\gamma_{\tau} \mathcal{C}_x^{1+\kappa}}\left( 1 + 2T^{\gamma \kappa}R^{1+\kappa}\right) \Big( (t-s)^\gamma s^\gamma + s^\gamma (t-s)^\gamma + (t-s)^{2\gamma} \Big)
        \\ &\leq \|A^b\|_{C^\gamma_{\tau} \mathcal{C}_x^{1+\kappa}}\left( 1 + 2T^{\gamma \kappa}R^{1+\kappa}\right) (t-s)^{\gamma} 3 \tau^\gamma
    \end{align*}
    since $s \leq \tau$ and $(t-s) \leq \tau$, and hence bounds the H\"older seminorm on its increments. Moreover, by $\mathcal{T}_0(\theta) = u_0$, we find $|\mathcal{T}_t(\theta) - u_0| \lesssim \|A^b\|_{C^\gamma_{\tau} \mathcal{C}_x^{1+\kappa}}\left( 1 + 2T^{\gamma \kappa}R^{1+\kappa}\right) T^{\gamma} 3 \tau^\gamma$, which yields 
    \[
        \|\mathcal{T}(\theta)\|_{C^{\gamma}([0,\tau])} \leq C_1 \|A^b\|_{C_{\tau}^\gamma \mathcal{C}_x^{1+\kappa}} \tau^\gamma \left( 1 + R^{1+\kappa} \right)
    \]
    for some constant $C_1 > 0$ independent of $\tau$ and $R$. 
    
    Next, let $\theta, \widetilde{\theta} \in \mathcal{B}_R$. The additional spatial regularity $A^b \in C_{\tau}^\gamma \mathcal{C}_x^{2+\kappa}$ allows us to apply the stability estimate from Lemma \ref{lemma two parameter nonlinear Young}. Noting that the first term therein vanishes, and that $\|\theta - \widetilde{\theta}\|_{C^{\gamma}([s_1,t_1])} + \|\theta - \widetilde{\theta}\|_{C^{\gamma}([s_2,t_2])} \leq \|\theta - \widetilde{\theta}\|_{C^{\gamma}([0,\tau])}$, we find 
    \begin{align*}
         \bigg| \int_{s_2}^{t_2}\int_{s_1}^{t_1}& A(\mathrm{d}r, \theta_{r_2} - \theta_{r_1}) - \int_{s_2}^{t_2}\int_{s_1}^{t_1}A(\mathrm{d}r, \widetilde{\theta}_{r_2} - \widetilde{\theta}_{r_1}) \bigg| 
        \\ &\lesssim \left( [\theta]_{C^{\gamma}([0,\tau])} + [ \widetilde{\theta}]_{C^{\gamma}([0,\tau])}\right)^{1+\kappa} \|A\|_{C^{\gamma}_{\tau} C_x^{2+\kappa}}\|\theta - \widetilde{\theta}\|_{C^{\gamma}([0,\tau])}(t_1-s_1)^{\gamma}(t_2-s_2)^{\gamma}
        \\ &\lesssim \left( 2R \right)^{1+\kappa} \|A\|_{C^{\gamma}_{\tau} C_x^{2+\kappa}}\|\theta - \widetilde{\theta}\|_{C^{\gamma}([0,\tau])}(t_1-s_1)^{\gamma}(t_2-s_2)^{\gamma}.
    \end{align*}
    Hence, we obtain 
    \begin{align*}
       |(\mathcal{T}_t(\theta) - \mathcal{T}_t(\widetilde{\theta})) - (\mathcal{T}_s(\theta) - \mathcal{T}_s(\widetilde{\theta}))|
        &\leq \left| \int_{s}^{t}\int_{0}^{s} A(\mathrm{d}r, \theta_{r_2} - \theta_{r_1}) - \int_{s}^{t}\int_{0}^{s}A(\mathrm{d}r, \widetilde{\theta}_{r_2} - \widetilde{\theta}_{r_1}) \right|
        \\ &\ + \left| \int_{0}^{s}\int_{s}^{t} A(\mathrm{d}r, \theta_{r_2} - \theta_{r_1}) - \int_{0}^{s}\int_{s}^{t}A(\mathrm{d}r, \widetilde{\theta}_{r_2} - \widetilde{\theta}_{r_1}) \right| 
        \\ &\ + \left| \int_{s}^{t}\int_{s}^{t} A(\mathrm{d}r, \theta_{r_2} - \theta_{r_1}) - \int_{s}^{t}\int_{s}^{t}A(\mathrm{d}r, \widetilde{\theta}_{r_2} - \widetilde{\theta}_{r_1}) \right| 
        \\ &\lesssim \left( 2R \right)^{1+\kappa} \|A\|_{C^{\gamma}_{\tau} C_x^{2+\kappa}}\|\theta - \widetilde{\theta}\|_{C^{\gamma}([0,\tau])} 3\tau^{\gamma} (t-s)^{\gamma}
    \end{align*}
    where we have used $s, t-s \leq \tau$. Since $\mathcal{T}_0(\theta) = \mathcal{T}_0(\widetilde{\theta})$, we conclude for the H\"older norm that
    \[
    \|\mathcal{T}(\theta) - \mathcal{T}(\widetilde{\theta})\|_{C^{\gamma}([0,\tau])} \leq C_2 \left( 2R \right)^{1+\kappa} \|A\|_{C^{\gamma}_{\tau} C_x^{2+\kappa}}\|\theta - \widetilde{\theta}\|_{C^{\gamma}([0,\tau])} \tau^{\gamma} 
    \]
    for some constant $C_2 > 0$.
    
    To close the argument, we first choose $R = 1$. Then, we select $\tau > 0$ sufficiently small such that
    \begin{align*}
        2C_1 \|A^b\|_{C_{\tau}^\gamma \mathcal{C}_x^{1+\kappa}} \tau^\gamma  \leq 1 \quad \text{and} \quad 2^{1+\kappa} C_2 \|A^b\|_{C_{\tau}^\gamma \mathcal{C}_x^{2+\kappa}} \tau^\gamma \leq \frac{1}{2}.
    \end{align*}
    With this choice, $\mathcal{T}$ maps $\mathcal{B}_1$ into itself and is a contraction mapping with Lipschitz constant bounded by $1/2$. By the Banach fixed-point theorem, there exists a unique solution $\theta \in C^\gamma([0,\tau]; \R^d)$. Since the choice of the step size $\tau$ depends only on the norm of $A^b$ and not on the initial condition, the solution can be iteratively extended over $[\tau, 2\tau], [2\tau, 3\tau]$, and so forth, yielding a unique global solution $u(t) = \theta_t + z_t$ on $[0,T]$.
\end{proof}

Finally, we can combine the local time approach with the above to obtain the existence, uniqueness, and stability of solutions in terms of the regularity of the two-parameter self-intersection measure.

\begin{Theorem}\label{thm: FT rough equation 2}
    Let $w \in L^{\infty}([0,T])$ and $z \in C([0,T]; \R^d)$. Suppose that $b \in \mathcal{F}L_{q'}^{\delta}(\R^d)$ and the two-parameter measure defined in \eqref{eq: selfintersection measure two parameter} satisfies $G \in C^{\gamma}([0,T]^2; \mathcal{F}L_q^{\kappa}(\R^d))$, where $\gamma \in (1/2,1)$, $\delta, \kappa \in \R$, and $q, q' \in [1,\infty]$ satisfy
    \begin{align}\label{eq: TH}
        \frac{1}{q} + \frac{1}{q'} = 1 \ \text{ and } \ \delta + \kappa > 1 + \frac{1}{\gamma}.
    \end{align} 
    Then \eqref{eq: selfintersecting Volterra diffusion process} has a unique solution $u$. This solution depends continuously on the drift $b$ and initial datum $u_0$. Namely, let $(b_n)_{n \geq 1} \subset \mathcal{F}L_{q'}^{\delta}(\R^d)$ be a sequence of Lipschitz continuous vector fields with $b_n \longrightarrow b$ in $\mathcal{F}L_{q'}^{\delta}(\R^d)$ and $(u_0^{(n)})_{n\geq 1} \subset \R^d$ with $u_0^{(n)} \longrightarrow u_0$. Denote by $u_n$ the unique solution of 
    \[
        u_n(t) = u^{(n)}_0 + \int_0^t \int_0^t w_s w_r b_n(u_n(s) - u_n(r))\mathrm{d}r\mathrm{d}s + z_t, \qquad t \in [0,T].
    \]
    Then there exists a constant $C > 0$ such that
    \[
        \|u_n - u\|_{C^{\gamma}([0,T])} \leq C\left(\|b_n - b\|_{\mathcal{F}L_{q'}^{\delta}} + |u_0^{(n)} - u_0| \right).
    \]
\end{Theorem}
\begin{proof}
    An application of Lemma \ref{lemma: Fourier Lebesgue embedding} and then Young's inequality \eqref{eq: Young FL}, gives for $s < t$ in $[0,T]^2$ the bound
    \begin{align*}
        \|b \ast \Box_{s,t} \overline{G}\|_{\mathcal{C}^{\delta + \kappa}}
        \lesssim \|b \ast \Box_{s,t} \overline{G}\|_{\mathcal{F}L_1^{\delta + \kappa}} 
        &\leq \| b\|_{\mathcal{F}L_{q'}^{\delta}} \| \Box_{s,t} \overline{G}\|_{\mathcal{F}L_q^{\kappa}}
    \\ &\leq \| b\|_{\mathcal{F}L_{q'}^{\delta}} \| \overline{G}\|_{C^{\gamma}([0,T]^2;\mathcal{F}L_q^{\kappa})}(t_2-s_2)^{\gamma}(t_1-s_1)^{\gamma},
    \end{align*}
    where the right-hand side is finite by assumption and since $G$ and $\overline{G}$ have the same regularity in the Fourier-Lebesgue space. Similarly, we obtain
    \begin{align*}
         \|b \ast \left(\overline{G}_{t_1,s_2} -  \overline{G}_{s_1,s_2}\right)\|_{\mathcal{C}^{\delta + \kappa}}
         \leq \| b\|_{\mathcal{F}L_{q'}^{\delta}} \| \overline{G}\|_{C^{\gamma}([0,T]^2;\mathcal{F}L_q^{\kappa})}(t_1-s_1)^{\gamma}.
    \end{align*}
    and analogously $\|b \ast \left(\overline{G}_{s_1,t_2} -  \overline{G}_{s_1,s_2}\right)\|_{\mathcal{C}^{\delta + \kappa}} \leq \| b\|_{\mathcal{F}L_{q'}^{\delta}} \| \overline{G}\|_{C^{\gamma}([0,T]^2;\mathcal{F}L_q^{\kappa})}(t_2-s_2)^{\gamma}$. This shows that $A^{b} \in C^{\gamma}([0,T]^2; \mathcal{C}^{\delta + \kappa}(\R^d))$. Since $\delta + \kappa > 2$ by assumption and $\gamma( 1 + (\delta + \kappa - 2)) = \gamma ( \delta + \kappa - 1) > 1$, the existence and uniqueness follow from Proposition \ref{thm: rough equation 2}.(b). For the stability of solutions, we may follow the same arguments as given in \cite[Proposition 23]{BHR23}, taking into account Lemma \ref{lemma two parameter nonlinear Young}.
\end{proof}

We close this section with a simple sufficient criterion for the regularity of  $G$ given by \eqref{eq: selfintersection measure two parameter}. 

\begin{Remark}\label{remark 5}
    Let $(Z_t)_{t \in [0,T]}$ be a stochastic process on $\R^d$. Suppose that there exists $\eta \in (0,1/2)$, $\kappa_* > 0$ and $p \geq 2$ such that the Fourier transform of the weighted occupation measure $\widehat{\ell}_{s,t}^{w}(\xi) = \int_s^t w_r \mathrm{e}^{\mathrm{i}\langle \xi, Z_r\rangle}\, \mathrm{d}r$ satisfies
    \[
        \left\| \widehat{\ell}_{s,t}^{w}(\xi) \right\|_{L^p(\Omega)} \lesssim (1+|\xi|)^{- \kappa_*}(t-s)^{1- \eta}, \qquad \xi \in \R^d.
    \]
    Then $G$ defined in \eqref{eq: selfintersection measure two parameter} with $z = Z$ satisfies for each $\e > 0$, $q \in [1,\infty)$, and $\kappa < 2\kappa_* - \frac{d}{q}$
    \[
        G \in L^{\frac{p}{2}}(\Omega, \P; C^{1 - \eta - \frac{2}{p} - \e}([0,T]^2; \mathcal{F}L_q^{\kappa}(\R^d))).
    \] 
\end{Remark}
\begin{proof}
    For given $(s_1, t_1), (s_2, t_2) \in \Delta_T^2$ we obtain by the Cauchy-Schwarz inequality
    \begin{align*}
        \left \| \Box_{s,t}\widehat{G}(\xi) \right\|_{L^{p/2}(\Omega)} 
        &= \left\| \int_{s_2}^{t_2} \int_{s_1}^{t_1} w_{r_2} w_{r_1} \mathrm{e}^{\mathrm{i}\langle \xi, Z_{r_2}- Z_{r_1}\rangle} \mathrm{d}r_1 \mathrm{d}r_2 \right\|_{L^{p/2}(\Omega)} 
        \\ &\leq \| \widehat{\ell}_{s_2, t_2}^{w}(\xi) \|_{L^{p}(\Omega)}\| \widehat{\ell}_{s_1, t_1}^{w}(-\xi)\|_{L^{p}(\Omega)}
        \\ &\lesssim (t_2 - s_2)^{1-\eta} (t_1 - s_1)^{1-\eta} (1+|\xi|)^{-2\kappa_*}.
    \end{align*}
    Moreover, we obtain
    \begin{align*}
        \left \| \widehat{G}_{t_1,s_2}(\xi) - \widehat{G}_{s_1,s_2}(\xi) \right\|_{L^{p/2}(\Omega)} 
        &= \left\| \int_{0}^{s_2}\int_{s_1}^{t_1} w_{r_2} w_{r_1} \mathrm{e}^{\mathrm{i}\langle \xi, Z_{r_2}- Z_{r_1}\rangle} \mathrm{d}r_1 \mathrm{d}r_2 \right\|_{L^{p/2}(\Omega)}
        \\ &\leq \left\| \widehat{\ell}_{0,s_2}^w(\xi)\right\|_{L^p(\Omega)} \left\| \widehat{\ell}^{w}_{s_1, t_1}(-\xi)\right\|_{L^p(\Omega)}
        \\ &\lesssim s_2^{1-\eta} (t_1 - s_1)^{1-\eta} (1+|\xi|)^{-2\kappa_*} 
        \\ &\leq T^{1-\eta} (t_1 - s_1)^{1-\eta} (1+|\xi|)^{-2\kappa_*},
    \end{align*}
    and in a completely analogous way, taking the interval $r_1 \in [0, s_1]$,
    \begin{align*}
        \left \| \widehat{G}_{s_1,t_2}(\xi) - \widehat{G}_{s_1,s_2}(\xi) \right\|_{L^{p/2}(\Omega)} 
        &= \left\| \int_{s_2}^{t_2}\int_{0}^{s_1} w_{r_2} w_{r_1} \mathrm{e}^{\mathrm{i}\langle \xi, Z_{r_2}- Z_{r_1}\rangle} \mathrm{d}r_1 \mathrm{d}r_2 \right\|_{L^{p/2}(\Omega)}
        \\ &\lesssim T^{1-\eta} (t_2 - s_2)^{1-\eta} (1+|\xi|)^{-2\kappa_*}.
    \end{align*}
    An application of the multi-parameter Kolmogorov continuity theorem as stated in \cite[Theorem 3.1]{MR3071383} to the bounds on $\Box_{s,t}G$ combined with the usual one-parameter Kolmogorov-Chentsov theorem applied to the remaining two bounds, yields the assertion.
\end{proof}

This remark is applicable, e.g. in the framework of Sections 3 and 4, see also \eqref{eq: Lp bound FT}. In such a case, the weight $w$ is given by $w_r = \rho_r^{\delta}$ with a suitable choice of $\delta$.

\subsection{Examples}

In this section, we collect a few examples that illustrate the applicability of our results. While our results apply to a general class of Volterra It\^{o} processes that satisfy assumptions (A1) -- (A3), below we focus on the most important case of the fractional Brownian motion $B^H$ on $\R^d$ with Hurst parameter $H \in (0,1)$. Its unweighted occupation measure satisfies the assumptions of Remark \ref{remark 5} with regularity $\kappa_*(\eta) = \frac{\eta}{H}$ and any choice of $p \geq 2$. In particular, for any choice of $\eta \in (0,1/2)$, $p \geq 2$, $0 < \gamma < 1 - \eta - \frac{2}{p}$, $q \in [1,\infty)$ and $\kappa < 2\kappa_*(\eta) - \frac{d}{q} = \frac{2\eta}{H} - \frac{d}{q}$, there exists $\Omega_0$ with $\P[\Omega_0] = 1$ such that its two-parameter self-intersection measure satisfies 
\begin{align}\label{eq: Example}
    G \in C^\gamma([0,T]^2; \mathcal{F}L_q^\kappa(\R^d)) \ \text{ on } \ \Omega_0.
\end{align}
Assume $b \in \mathcal{F}L_{q'}^{\delta}(\R^d)$, then Theorem \ref{thm: FT rough equation 2} is applicable, provided that \eqref{eq: TH} holds, i.e. $\delta + \kappa > 1 + 1/\gamma$. By choosing $\eta$ close to $\frac{1}{2}$, $p \geq 2$ sufficiently large, and $\kappa$ close to $2\kappa_*(\eta) - d/q$, gives $\gamma \approx 1/2$, and yields
\begin{align}\label{Example condition}
    \delta + \frac{1}{H} - \frac{d}{q} > 3.
\end{align}
Hence, if \eqref{Example condition} is satisfied, we may choose $\eta, p, \kappa$, find the corresponding $\Omega_0$ with property \eqref{eq: Example}, and finally apply Theorem \ref{thm: FT rough equation 2}.

Our first example provides an analogue of the fractional skew Brownian motion, with the difference that this process avoids its own sample paths rather than the origin.

\begin{Example}[self-interacting fractional Brownian motion]
    Let $d = 1$. If $H < 1/4$, then there exists $\Omega_0$ with $\P[\Omega_0] = 1$ such that for each $\omega \in \Omega_0$ the equation
    \[
        u(t) = u_0 + \int_0^t \int_0^t \delta_0(u(s) - u(r))\,\mathrm{d}r\mathrm{d}s + B_t^H(\omega)
    \]
    has a unique solution.
\end{Example}
\begin{proof}
    Here $b = \delta_0$ and hence $\widehat{b}(\xi) = 1$. Thus, $b \in \mathcal{F}L_{q'}^\delta(\R)$ for $q' = \infty$ and $\delta = 0$. Hence $q = 1$ and \eqref{Example condition} reduces to $H < 1/4$. This proves the assertion. 
\end{proof}

In our next example, we consider a dynamical version of the Edwards model for polymer physics. 

\begin{Example}[continuous Edwards model]
    Let $\theta \in \R \backslash \{0\}$. If $H < \frac{1}{d+4}$, then there exists $\Omega_0$ with $\P[\Omega_0] = 1$ such that for each $\omega \in \Omega_0$ the equation
    \[
        u(t) = u_0 + \theta \int_0^t \int_0^t \nabla \delta_0(u(s) - u(r))\, \mathrm{d}r\mathrm{d}s + B_t^H(\omega)
    \]
    has a unique solution.
\end{Example}
\begin{proof}
    In this case we find $b = \theta \nabla \delta_0$, whose Fourier transform is given by $\widehat{b}(\xi) = \mathrm{i}\xi$. This implies $b \in \mathcal{F}L_{\infty}^{-1}(\R^d)$, i.e. $q' = \infty$ and $\delta = -1$. Hence $q = 1$ and condition \eqref{Example condition} reduces to $1/H > 4+d$, which proves the assertion. 
\end{proof}

The flexibility of the Fourier-Lebesgue spaces allows us to easily treat fractional singularities. 

\begin{Example}[continuous Edwards model with fractional interactions]
    Let $B^H$ be the fractional Brownian motion on $\R^d$. Let $b(x) = \nabla \left( |x|^{-\alpha} \chi(x) \right)$ where $\alpha \in (0, d-1)$ and $\chi: \R^d \longrightarrow \R_+$ is smooth, compactly supported, and satisfies $\chi(x) = 1$ in a neighbourhood of the origin. If 
    \[
        H < \frac{1}{4+\alpha},
    \]
    then there exists $\Omega_0$ with $\P[\Omega_0] = 1$ such that for each $\omega \in \Omega_0$ the equation
    \[
        u(t) = u_0 + \theta \int_0^t \int_0^t \nabla \left( |u(s) - u(r)|^{-\alpha} \chi(u(s) - u(r)) \right)\, \mathrm{d}r\mathrm{d}s + B_t^H(\omega)
    \]
    has a unique solution.
\end{Example}
\begin{proof}
    Let $f(x) = \theta |x|^{-\alpha}\chi(x)$ so that $b(x) = \nabla f(x)$. Since $\alpha \in (0,d-1)$, it follows from \cite[Theorem 2.4.6]{MR3243734} that the Fourier transform of $b_0(x) = |x|^{-\alpha}$ is given by $\widehat{b}_0(\xi) = C_{d,\alpha}|\xi|^{\alpha - d}$, where $C_{d,\alpha} = \frac{\pi^{\alpha - \frac{d}{2}}\Gamma\left( \frac{d-\alpha}{2}\right)}{\Gamma\left(\frac{\alpha}{2}\right)}$. Because $f$ is compactly supported and possesses an integrable singularity of order $\alpha$ at the origin, its Fourier transform is well-defined. Since $\widehat{\chi} \in \mathcal{S}(\R^d)$ due to the compact support of $\chi$, we obtain
    \[
        \widehat{f}(\xi) = (\widehat{\chi} \ast \widehat{b}_0)(\xi) = \int_{\R^d}\widehat{\chi}(\xi - z)C_{d,\alpha}|z|^{\alpha - d}\, \mathrm{d}z \sim |\xi|^{\alpha - d}, \qquad \text{as } |\xi| \to \infty.
    \]
    Because $b(x) = \nabla f(x)$, its Fourier transform is given by $\widehat{b}(\xi) = i\xi \widehat{f}(\xi)$. Consequently, $|\widehat{b}(\xi)| \sim |\xi|^{\alpha - d + 1}$ as $|\xi| \to \infty$. Using this asymptotics, we find $b \in \mathcal{F}L_{q'}^\delta(\R^d)$ provided that $\delta + (\alpha + 1 - d) < -d/q'$, where $q' \in [1,\infty)$. Let $q \in (1,\infty]$ be the conjugate exponent of $q'$, so that $d/q' = d - d/q$. This allows us to choose any $\delta$ strictly satisfying $\delta < d/q - \alpha - 1$. Plugging the upper bound on $\delta$ into condition \eqref{Example condition} yields 
    \[
        \left( \frac{d}{q} - \alpha - 1 \right) + \frac{1}{H} - \frac{d}{q} > 3,
    \]
    which reduces to $H < \frac{1}{4+\alpha}$, and proves the assertion.
\end{proof}

As a final example, we consider a self-interacting diffusion equation driven by fractional noise, which can be viewed as an integrated, fractional analogue of the classical Durrett-Rogers model for self-repelling polymer diffusions.

\begin{Example}[generalised Durrett-Rogers model]
    Suppose that $d = 1$. Let $\theta \neq 0$ and  $\chi: \R \longrightarrow \R_+$ be smooth, compactly supported, and satisfy $\chi(x) = 1$ in a neighbourhood of the origin. If $H < 1/3$, then there exists $\Omega_0$ with $\P[\Omega_0] = 1$ such that for each $\omega \in \Omega_0$ the equation
    \[
        u(t) = u_0 + \theta\int_0^t \int_0^t \chi(u(s) - u(r)) \mathrm{sgn}(u(s) - u(r))\, \mathrm{d}r\mathrm{d}s + B_t^H(\omega)
    \]
    has a unique solution.
\end{Example}
\begin{proof}
    In this case, $b(x) = \theta \chi(x)\mathrm{sgn}(x)$ acts as a localised step function. To find the asymptotics of its Fourier transform, let us first note that its distributional derivative is given by $b'(x) = 2 \theta \delta_0(x) + \theta \chi'(x)\mathrm{sgn}(x)$ since $\chi(0) = 1$. Taking the Fourier transform gives $\widehat{b'}(\xi) = \mathrm{i} \xi \widehat{b}(\xi) = 2\theta + \widehat{\varphi}(\xi)$ where $\widehat{\varphi} \in \mathcal{S}(\R)$ since $\varphi(x) = \theta \chi'(x)\mathrm{sgn}(x)$ is compactly supported and smooth as $\chi'$ vanishes in a neighbourhood of the origin. Hence we obtain $|\widehat{b}(\xi)| \sim |\xi|^{-1}$ as $|\xi| \to \infty$. Thus, $b \in \mathcal{F}L_{q'}^\delta(\R)$ provided that $\delta < 1 - 1/q'$ and $q' \in [1,\infty)$. Set $q \in (1,\infty]$ by $1/q' = 1 - 1/q$, then $\delta < 1/q$. Inserting this upper bound for $\delta$ into condition \eqref{Example condition} with $d=1$ yields $H < 1/3$, and proves the assertion.
\end{proof}

\subsection*{Acknowledgements}
The author would like to thank Kristof Wiedermann for pointing out a necessary extension of conditions (A3) and (B2) that allows for a larger class of models. 

\appendix

\bibliographystyle{amsplain}
\phantomsection\addcontentsline{toc}{section}{\refname}\bibliography{Bibliography}

\end{document}